\documentclass[a4paper]{amsart}
\usepackage{amsmath, amssymb, amsthm, amsfonts}
\usepackage{xcolor}

\usepackage{paralist}
\usepackage{listings}
\usepackage{xspace}

\usepackage[ngerman,british]{babel}

\usepackage[style=alphabetic,
   sorting=nty,
   backend=biber,
   maxbibnames=99,
   doi=false,
   isbn=false,
   backref=false]{biblatex}
\usepackage[tracking, kerning, spacing]{microtype}
\microtypecontext{spacing=nonfrench}

\usepackage{tikz}
\usepackage{tikz-3dplot}
\usetikzlibrary{matrix,fit,positioning,shapes.geometric,patterns,backgrounds}
\usetikzlibrary{decorations.pathreplacing}
\usetikzlibrary{arrows,calc}
\tikzstyle{bigbox} = [draw=blue!50, thick, rounded corners, rectangle]
\tikzset{
>=stealth'
}

\usepackage{algorithm}
\usepackage{algpseudocode}

\usepackage[normalem]{ulem}

\newcommand{\polymakejl}{{\texttt{Polymake.jl}}\xspace}
\newcommand{\polymake}{{\texttt{polymake}}\xspace}
\newcommand{\cellularSheaves}{{\texttt{cellularSheaves}}\xspace}
\newcommand{\julia}{{\texttt{Julia}}\xspace}
\newcommand{\jupyter}{{\texttt{Jupyter}}\xspace}

\usepackage{aliascnt}

\usepackage{hyperref}

\newcounter{internal}[section]

\newaliascnt{intcor}{internal} 
\newaliascnt{intconj}{internal} 
\newaliascnt{intlemma}{internal} 
\newaliascnt{intdef}{internal} 
\newaliascnt{intex}{internal} 
\newaliascnt{intprop}{internal} 
\newaliascnt{intrem}{internal} 
\newaliascnt{intthm}{internal}

\newcommand*{\dupcntr}[2]{%
   \expandafter\let\csname c@#1\expandafter\endcsname\csname c@#2\endcsname
}
\dupcntr{figure}{internal}

\theoremstyle{definition}

\newtheorem{lemma}[intlemma]{Lemma}
\newtheorem{definition}[intdef]{Definition}
\newtheorem{example}[intex]{Example}
\newtheorem{prop}[intprop]{Proposition}

\theoremstyle{remark}
\newtheorem{remark}[intrem]{Remark}

\theoremstyle{plain}
\newtheorem{theorem}[intthm]{Theorem}

\newcommand{\macrocolor}[1]{{#1}}

\DeclareMathOperator{\cl}{\mathop{\macrocolor{cl}}} %
\DeclareMathOperator{\cone}{\macrocolor{cone}}
\DeclareMathOperator{\conv}{\macrocolor{conv}}
\DeclareMathOperator{\edges}{\macrocolor{edges}}
\DeclareMathOperator{\hasse}{\macrocolor{HD}}
\DeclareMathOperator{\Hom}{\macrocolor{Hom}}
\DeclareMathOperator{\minface}{\macrocolor{minface}}
\DeclareMathOperator{\minfacevert}{\macrocolor{minfacevert}}
\DeclareMathOperator{\nodes}{\macrocolor{nodes}}
\DeclareMathOperator{\signedIR}{\macrocolor{\mathcal{OR}}}
\DeclareMathOperator{\parent}{\macrocolor{parent}}
\DeclareMathOperator{\realisation}{\macrocolor{rls}}
\DeclareMathOperator{\relint}{\macrocolor{relint}}
\DeclareMathOperator{\trunk}{\macrocolor{trunk}}
\DeclareMathOperator{\sedentarity}{\macrocolor{sed}}
\DeclareMathOperator{\support}{\macrocolor{support}}
\DeclareMathOperator{\vertices}{\macrocolor{vert}}
\DeclareMathOperator{\vsspan}{\macrocolor{span}}

\DeclareMathOperator{\tail}{\macrocolor{rec}}

\DeclareMathOperator{\CP}{\mathbb{C}P}

\newcommand{\disjoint}{\macrocolor{\sqcup}}
\newcommand{\Disjoint}{\mathop{\macrocolor{\coprod}}}
\newcommand{\newVert}{\macrocolor{\mathcal A}}
\newcommand{\pc}{\macrocolor{\mathcal P\mathcal C}}
\newcommand{\powerset}{\macrocolor{\mathcal P\mathcal S}}
\newcommand{\poly}{\macrocolor{P}}
\newcommand{\RR}{\macrocolor{\mathbb R}}

\newcommand{\compact}[1]{\macrocolor{\overline{#1}}}
\newcommand{\compactifying}{\macrocolor{N_{\RR}}}

\newcommand{\sigmamin}{{\macrocolor{t}}}
\newcommand{\surj}{\macrocolor{\twoheadrightarrow}}
\newcommand{\define}[1]{\macrocolor{\emph{#1}}}

\title{Tropical compactification via Ganter's algorithm}
\author{
Lars Kastner,
Kristin Shaw,
Anna-Lena Winz
}
\date{}

\begin{document}

\begin{abstract}

We describe a canonical compactification of a polyhedral complex in Euclidean space. 
When the recession cones of the polyhedral complex form a fan, the compactified polyhedral complex 
is a subspace of a tropical toric variety. In this case, the procedure is analogous to the tropical compactifications 
of subvarieties of tori. 

We give an analysis of the combinatorial structure of the compactification and show that its Hasse diagram can be computed via Ganter's algorithm. 
Our algorithm is implemented in and shipped with \polymake. 

\end{abstract}
\maketitle
\section{Introduction}

Most of the first steps in tropical geometry considered the tropicalisation of
subvarieties of tori \cite{FirstSteps}.  Yet early on Mikhalkin considered the
compactification of tropical curves and amoebas in toric surfaces \cite{Amoebas}, and then  of
hypersurfaces in toric varieties  \cite{PairsOfPants}. Later Payne described ``extended
tropicalisations" for subvarieties of toric varieties \cite{Payne}. 
These extended
tropicalisations  apply the tropicalisation of subvarieties of tori in
an orbit by orbit fashion.  
Previous to this, Tevelev  used tropicalisations to define well-behaved compactifications 
of classical algebraic subvarieties in the torus \cite{TevTropComp}. In this setup, the tropicalisation of a subvariety of the torus determines a toric variety suitable for the compactification of the original variety. 

Our main goal here is to analyse combinatorially and computationally the  structure of  canonical compactifications of tropical varieties. More concretely, given  a tropical variety in Euclidean space,  we describe a canonical compactification of the underlying  polyhedral complex which is compatible with the compactifications of single polyhedra described in \cite{Rabinoff}. 
To define the canonical compactification of a tropical variety, all geometric
data needed is  encoded in a choice of polyhedral structure on the original tropical variety. 

Throughout we let $N$ be a lattice and $N_{\RR} = N \otimes_{\mathbb{Z}} \RR$.
Given a polyhedron $\poly = \{ x \in N_{\RR}  | A x \geq b \}$, its  recession cone is $\tail(\poly)  = \{ x \in N_{\RR} | \ Ax \geq 0\}$.

\begin{definition}\label{def:pc_has_rec_fan}
Let $\pc$ be a polyhedral complex in $N_{\RR}$. We say that $\pc$ \define{has a recession fan}
if for any $\poly,Q\in\pc$ the intersection $\tail(\poly)\cap\tail(Q)$ is a face of
both $\tail(\poly)$ and $\tail(Q)$. In that case the recession cones of the
polyhedra in $\pc$ form a fan, and we call this fan the \define{recession fan of
$\pc$}, denoted by $\tail(\pc)$.
\end{definition}

A polyhedral complex may or may not have a recession fan \cite{recessionFan}.
In the case when it does, we prove the following theorem. %

\begin{theorem}[\autoref{prop:pc}] 
If a polyhedral complex $\pc$ has a recession fan, then its canonical compactification 
 $\compact{\pc}$ is a polyhedral complex in  the tropical toric variety of the fan $\tail(\pc)$.
\end{theorem}

By \cite[Remark~3.13]{ossermanRabinoff} any polyhedral complex \(\pc\) has a
compactifying fan which can be obtained by choosing a fan that is a refinement
of the set of cones   \(\cup_{\poly \in \pc} \tail(\poly)\). %
 In fact,  if we consider a polyhedral
complex $\pc$ and an arbitrary fan $\Delta$, many of our algorithmic results
are still applicable if  $\Delta$ refines all the cones $\tail(\poly)$ for
$\poly\in\pc$, see Definition \ref{def:fan_compatible}. 
However, even if  $\tail(\pc)$ is not a fan we can nonetheless define a canonical
compactification  of $\pc$ without taking any refinements, as long as the
recession cone of each face is pointed. However, the resulting compactification
is not a polyhedral complex in a tropical toric variety. It is a more general abstract
polyhedral space in the sense of \cite[Definition 2.1]{JellRauShaw} or
\autoref{def:polyhedralspace}.

\begin{theorem}[\autoref{prop:polyspace}]
The canonical compactification $\compact{\pc}$ is an abstract polyhedral space. 
\end{theorem}

Here we  will describe the compactification $\compact{\pc}$ via the Hasse
diagram of its face lattice. A Hasse diagram is a graphical representation of a partially ordered set. 
The  vertices, also known as nodes, of the Hasse diagram correspond to elements of the set and 
the edges correspond to covering relations, with edges being directed upwards, towards the larger sets. 
One very
efficient algorithm for computing Hasse diagrams or general closure systems is
Ganter's algorithm  \cite{ganter_orig}.  To use Ganter's algorithm we have to construct a closure
operator for our setting.  This closure operator takes any subset of vertices to
the smallest face containing it.  
Therefore, applying Ganter's algorithm first necessitates the expression of  the
vertices and faces of the  compactification in the data encoding the original
polyhedral complex.

Our algorithmic results give rise to a closure operator in
\autoref{prop:closure_operator}. 
For the concrete problem of determining $\compact{\pc}$
it turns out that no new geometric information
is needed, hence we can state the following theorem.

\begin{theorem}
The Hasse diagram of the compactification $\compact{\pc}$ 
can be computed via Ganter's algorithm in a
purely combinatorial way.
\end{theorem}

Our implementation is shipped with the combinatorial software framework
\polymake \cite{polymake} since release 3.6, and hence it is available via many package
managers on Linux and in the MacOS \polymake bundle. It can even be used on
Windows via the \texttt{Windows Subsystem for Linux} (WSL). Furthermore,
\polymake is interfaced in \julia \cite{julialang} via \polymakejl
\cite{polymakejl}. Hence our algorithm is accessible to a large community of
mathematicians, namely the users of \polymake and \julia, and it is embedded in
frameworks that provide a wide variety of tools for analysing and using the
tropical compactification. By using \polymake we took advantage of its existing
templated version of Ganter's algorithm by Simon Hampe and Ewgenij Gawrilow.
The necessary data types, like Hasse diagrams, polyhedral complexes, tropical
varieties, and chain complexes are already implemented in \polymake, making our
codebase slim and easy to maintain. For using the compactification in
subsequent research, \polymake already has cellular sheaves \cite{ourTropHom} and patchworkings
\cite{paul}, as well as many other tools from tropical geometry. Last, but not
least, \polymake comes with a built in serialization framework, such that the
compactification can easily be stored in a file, and testing framework,
ensuring robustness of our implementation. 

In fact, our main motivation  to study canonical compactifications is to
extend the use of the \polymake extension \cellularSheaves for tropical
homology and patchwork to compact tropical varieties.
In the case of  tropicalisations of projective complex
varieties satisfying additional assumptions, the dimensions of the tropical
homology groups are equal to the corresponding Hodge numbers. The assumption
that the variety is projective implicitly assumes that the tropicalisation
under consideration is an extended tropicalisation in the sense of Payne or
Mikhalkin, and is hence compact. Therefore, the Hasse diagram of the
compactification is necessary for such computations, together with a signed incidence
relation that we describe in \autoref{sec:sir}.  

In \autoref{sec:prelim} we will give the necessary definitions from tropical
and algorithmic geometry for our setup. Afterwards in
\autoref{sec:face_lattice} we describe the Hasse diagram
of the compactification
for both a single polyhedron and a polyhedral complex. The data structure of the compactification is described in \autoref{sec:datastr}.  In \autoref{sec:sir}, we
give a simple algorithm for computing a signed incidence relation necessary for
computing cohomology of cellular sheaves. Lastly, \autoref{sec:polymake}
contains examples with code computed in \polymake.
Throughout the text we emphasis many examples which exhibit the pathologies of
the canonical compactification, as well as its many applications to tropical geometry and beyond.

\subsection{Acknowledgements}
We are very grateful to Michael Joswig for advice throughout the implementation
and for suggesting the connection to Ganter's algorithm.
We are very grateful to Benjamin Lorenz for advice on designing our codebase
and for helping to solve many complex implementation specific issues.

We also thank Joswig, Marta Panizzut, and Paul Vater for their comments which helped us
improve a preliminary draft of this paper. 

\section{Preliminaries}\label{sec:prelim}

\subsection{Tropical toric varieties and polyhedral complexes}

In this section we will describe the basic setup, following the definitions and
notation of \cite{ossermanRabinoff}, \cite{Payne}, \cite{MikRau}. %
Throughout we let $N \cong \mathbb{Z}^n$ denote a lattice and $N_{\RR} = N \otimes_{\mathbb{Z}} \RR$. 

\begin{definition}\cite[2.4]{ossermanRabinoff}
For a
rational polyhedral fan $\Delta\subseteq N_{\RR}$ the 
\define{tropical toric variety}
$\compactifying(\Delta)$ of $N_{\RR}$ with respect to $\Delta$ 
is
\[
\compactifying(\Delta)\ :=\ \Disjoint_{\sigma\in\Delta}N_{\RR}/\vsspan(\sigma).
\]
\end{definition}
The tropical toric variety   $\compactifying(\Delta)$ is  equipped with
the unique topology such that 
\begin{itemize}
\item The inclusions $N_{\RR}/\vsspan(\sigma) \hookrightarrow \compactifying(\Delta)$ are continuous for any cone $\sigma \in\Delta$.
\item For any $x \in N_{\RR}$ and any $v \in N_{\RR}$, the sequence $\left(x+nv\right)_{n\in\mathbb{N}}\in N_{\RR}$ converges in $ \compactifying(\Delta)$ if and only if $v$ is contained in the support of the fan $\Delta$.
\end{itemize}
The reader is directed to 
\cite[Section 3]{Payne} for more details.
The tropical toric variety $\compactifying(\Delta)$ is compact if and only if
the polyhedral fan $\Delta$ is complete.
We denote the single stratum $N_{\RR}/\vsspan(\sigma)$ by
$\compactifying(\sigma)$.
If \(\Delta\) is a pointed fan, then \(N_{\RR}\) can be canonically identified with
the open subset \(\compactifying({0}) \subset \compactifying(\Delta)\).

\begin{definition}\label{def:pc}
A \define{polyhedral complex} $\pc$  in a tropical toric variety
$\compactifying(\Delta)$ is a finite collection of polyhedra in
$\compactifying(\Delta)$ such that   $\pc \cap \compactifying(\rho)$ is   a
polyhedral complex in $\compactifying(\rho) \cong \mathbb{R}^{\text{codim}
\rho}$ for every cone $\rho $ of $\Delta$, and which satisfies: 
\begin{enumerate}
\item for a polyhedron $\poly \in \pc$, if $Q$ is a face of $\poly$, which
is denoted $Q \le \poly$, we have $Q \in \pc$; 
\item  for $\poly, \poly' \in \pc$, if $Q = \poly \cap \poly'$ is
non-empty then $Q$ is a face of both $\poly$ and $\poly'$. 
\end{enumerate}
\end{definition}

\begin{definition}\cite[Definition~3.1]{ossermanRabinoff}\label{def:fan_compatible}
Let \(\mathcal{\poly}\) be a finite collection of polyhedra in \(N_{\RR}\), 
and \(\Delta\) a pointed fan. 
The fan \(\Delta\) is said to be \define{compatible with \(\mathcal{\poly}\)} if
for all \(\poly \in \mathcal{\poly}\) and all cones \(\sigma \in \Delta\),
either \(\sigma \subset \tail(\poly)\) or \(\relint(\sigma) \cap \tail(\poly) = \emptyset\).

The fan \(\Delta\) is said to be a \define{compactifying fan for \(\mathcal{\poly}\)}
if for all \(\poly \in \mathcal{\poly}\), the recession cone is the union of cones in 
\(\Delta\).
\end{definition}

\begin{example}\label{example:fan_incompatible}
For an example of $\Delta$ being incompatible with $\poly$ consider the following polyhedra:
\[
\begin{array}{ccc}
\poly = 
\begin{tikzpicture}[baseline=.5cm]
\draw[color=black!40] (-1.3,-.3) grid (1.3,1.3);
\draw[thick] (-1.3,1.3) -- (0,0) -- (1.3,1.3);
\fill[pattern color=black!40, pattern=dots] (-1.3,1.3) -- (0,0) -- (1.3,1.3) -- cycle;
\end{tikzpicture}
& \mbox{ and } &
\Delta = 
\begin{tikzpicture}[baseline=.5cm]
\draw[color=black!40] (-.3,-.3) grid (1.3,1.3);
\draw[thick] (0,1.3) -- (0,0) -- (1.3,0);
\fill[pattern color=black!40, pattern=dots] (0,1.3) -- (0,0) -- (1.3,0) -- (1.3,1.3) -- cycle;
\end{tikzpicture}
\end{array}.
\]
In this case, $\poly=\tail(\poly)$. The fan $\Delta$ has only one
maximal cone which intersects $\tail(\poly)$ improperly. This example serves
for us to show what goes wrong when $\poly$ and $\Delta$ are incompatible.
\end{example}

Given a polyhedron $\poly \subset N_{\RR}$ we can take its closure
$\compact{\poly}$ in $\compactifying(\Delta)$. 
Following \cite{ossermanRabinoff}, whether or not $\compact{\poly}$ intersects
a stratum of $\compactifying(\Delta)$ corresponding to a cone $\sigma$ of
$\Delta$ depends on the recession cone of $\poly$. 

To explicitly describe the intersection of $\compact{\poly}$ with each stratum
of $\compactifying(\Delta)$ we must consider the  projections
$\pi_{\sigma}:N_{\RR}\surj \compactifying(\sigma)$ for every
$\sigma\in\Delta$. Then when the intersection $\compact{\poly} \cap
\compactifying(\sigma)$ is non-empty, it is equal to $\pi_{\sigma}(\poly)$.
These statements are summarised in the following lemma from
\cite{ossermanRabinoff}. 
\begin{lemma}[{\cite[Lemma 3.9]{ossermanRabinoff}}]\label{def:compactification}
  Let $\Delta$ be a compactifying fan of $\poly$.  The compactification $\compact{\poly}$ of a polyhedron $\poly$ in
   $\compactifying(\Delta)$ is
   \[
   \compact{\poly}\ :=\ \Disjoint_{\sigma\in\Delta, \  \relint(\sigma)\cap\tail(\poly)\not= \emptyset} \pi_{\sigma}(\poly).
   \]
\end{lemma}

The main condition to ensure that \(\compact{\poly}\) is indeed compact is that $\Delta$ 
is a compactifying fan for \(\poly\), meaning that it refines the recession cone of $\poly$. In other words, 
we have
\begin{equation}
\tail(\poly)\ =\ \bigcup_{\sigma\in\Delta,\ \relint(\sigma)\cap\tail(\poly)\not= \emptyset} \sigma. \label{condition:covering}
\end{equation}
Then $\compact{\poly} \cap \compactifying(\sigma) \neq \emptyset $ if and only if 
\(\tail(\poly) \cap \relint(\sigma) \neq \emptyset\).

\begin{remark}\label{lemma:P_pc_0}\cite[Corollary~3.7]{ossermanRabinoff}
The intersection of the compactification with a stratum 
\(\compact{\poly} \cap \compactifying(\sigma)\) 
is the polyhedron \(\pi_{\sigma}(\poly)\) whenever the intersection is non-empty.
Notice that the intersection of \(\compact{\poly}\) with a stratum \(\compactifying(\sigma)\)
is not necessarily compact.  That is no surprise, since the intersection with the stratum
\(\compactifying(0)\) is the (non-compact) polyhedron \(\poly\) that
we started with. 
\end{remark}

\begin{example}[{\cite[Example~3.20]{Rabinoff}}]\label{example:pos_orthant}
Consider the following $\poly$ having the positive orthant as recession cone:
\[
\begin{tikzpicture}
\draw[color=black!40] (-.3,-.3) grid (4.3,4.3);
\draw[thick] (0,4.3) -- (0,1) -- (1,0) -- (4.3,0);
\draw[-stealth, thick] (0,1) -- (0,4.3);
\draw[-stealth, thick] (1,0) -- (4.3,0);
\fill[pattern color=black!40, pattern=dots] (0,4.3) -- (0,1) -- (1,0) -- (4.3,0) -- (4.3,4.3) -- cycle;
\path[-stealth] (4.5,2) edge node[below]{$\pi_{\RR_{\ge 0}\times\{0\}}$} (6.5,2);
\path[-stealth] (2,4.5) edge node[right]{$\pi_{\{0\}\times\RR_{\ge 0}}$} (2,5.5);
\path[-stealth] (4.5,4.5) edge node[below right]{$\pi_{\RR_{\ge 0}^2}$} (6.5,5.5);
\draw[color=black!40] (7,-.3) -- (7,4.3);
\draw[-stealth, thick] (7,0) -- (7,4.3);
\draw[color=black!40] (-.3,6) -- (4.3,6);
\draw[-stealth, thick] (0,6) -- (4.3,6);
\fill (7,6) circle (2pt);
\fill (1,0) circle (2pt);
\fill (0,1) circle (2pt);
\fill (7,0) circle (2pt);
\fill (0,6) circle (2pt);
\end{tikzpicture}
\]
The compactification has five vertices indicated by dots. For the
compactification we chose $\Delta=\tail(\poly)$, the fan having the recession
cone of $\poly$ as single maximal cone.
\end{example}

From now on we assume that \(\Delta\) has no lineality to avoid effects like in the following example.
\begin{example} \label{example:lineality}
Let \(\poly = \RR\), then \(\Delta = \tail(\poly)= \RR\), and  it is not a pointed fan. 
Then 
\(\compactifying(\Delta) = \compact{\poly} = \pi_{\RR}(\RR)\)  
which is just a point. 
Since \(0 \) is not a cone of \( \Delta\) the set \(N_{\RR}\) cannot even be seen as an open subset of \(\compactifying(\Delta)\).
\end{example}

Lastly, we give the definition of abstract polyhedral space from   \cite{JellShawSmacka} and \cite{JellRauShaw}.  This describes the structure of the compactification when the recession cones of a polyhedral complex do not form a fan.  Here  $\mathbb{T} := [-\infty, \infty)$ and is equipped with the topology of the half open interval. %
Notice that $\mathbb{T}^r = N_{\RR}(\Delta)$, where $\Delta$ is the cone in $\RR^r$ generated by the $r$ standard basis vectors. Hence it is a tropical toric variety.

\begin{definition}
\label{def:polyhedralspace}
A \emph{polyhedral space} $X$ is a paracompact, second countable Hausdorff topological space with an atlas of charts 
$(\varphi_{\alpha} \colon U_{\alpha} \rightarrow \Omega_{\alpha} \subset X_{\alpha})_{{\alpha} \in A}$ such that:
\begin{enumerate}
\item \label{def:charts}
The $U_{\alpha}$ are open subsets of $X$, the $\Omega_{\alpha}$ are open subsets of polyhedral subspaces 
$X_{\alpha} \subset \mathbb{T}^{r_{\alpha}}$, and the maps 
$\varphi_{\alpha} \colon U_{\alpha} \rightarrow \Omega_{\alpha}$ are homeomorphisms for all $\alpha$; 
\item \label{def:trans}
for all $\alpha, \beta \in A$ the transition maps
\begin{align*}
\varphi_{\alpha}\circ \varphi^{-1}_ \beta\colon \varphi_ \beta(U_{\alpha} \cap U_ \beta) \rightarrow {\varphi_\alpha(U_{\alpha} \cap U_ \beta)}
\end{align*}
are  extended affine linear maps, see \cite[Definition 2.18]{JellShawSmacka}.
\end{enumerate} 
\end{definition}

\subsection{Ganter's algorithm}\label{sec:ganter}
We want to use Ganter's algorithm for computing closure systems in order to
compute the Hasse diagram $\hasse(\compact{\pc})$ of the compactification
$\compact{\pc}$. We will follow the notation of \cite{ganter}, the original
work by Ganter can be found in \cite{ganter_orig}, which is an English reprint
of a German preprint from 1984. The input for Ganter's algorithm is a closure system.

\begin{definition}[{\cite[Def. 2.1]{ganter}}]\label{def:closure_operator}
   A \define{closure operator} on a set $S$ is a function $\cl: \mathcal{\powerset}(S) \to
   \mathcal{\powerset}(S)$
   on the power set of $S$, which fulfills the following axioms
   for all subsets $A,B\subseteq S$:
   \begin{enumerate}
      \item[(i)] $A \subseteq \cl(A)$ (Extensiveness).
      \item[(ii)] If $A \subseteq B$ then $\cl(A) \subseteq \cl(B)$ (Monotonicity).
      \item[(iii)] $\cl(\cl(A)) = \cl(A)$ (Idempotency).
   \end{enumerate}
   A subset $A$ of $S$ is called \define{closed}, if $\cl(A) = A$.
   The set of all closed sets of $S$ with respect to some closure operator is
   called a \define{closure system}.
\end{definition}

\begin{example}\label{example:closure_operator}
For a polytope $\poly$, the set $S$ would be the vertices $\vertices(\poly)$
and the closure operator $\cl(A)$ for $A\subseteq S$ would list the vertices of
the smallest face containing $A$. This is also the approach we want to use in
our setting.

This use of the closure operator also works for other combinatorial objects,
like cones, fans, polyhedral complexes, and flats of matroids, see
\autoref{sec:comp_pc}.

\end{example}

The idea of Ganter's algorithm is to start out with the empty set and then
successively add vertices until the full closure system is computed. The
algorithm is designed in a way that it is output sensitive, i.e. its running
time is linear in the number of edges of the Hasse diagram
$\hasse(\compact{\pc})$.

We will first solve the below steps for the case of our polyhedral complex
consisting of a single polyhedron. Afterwards we argue why this extends
seamlessly to polyhedral complexes whose recession cones form a fan.
\begin{enumerate}
\item Determine what the vertices $\vertices(\compact{\pc})$ should be. %
\item State an algorithm determining whether a subset of
$\vertices(\compact{\pc})$ forms a face.
\item Describe the closure operator on $\vertices(\compact{\pc})$ algorithmically.
\end{enumerate}

The first two tasks are mainly about rephrasing existing mathematical concepts
on $\compact{\pc}$ in combinatorial, and later in algorithmic terms. The third
task could also be solved in a brute force manner as soon as the second task is
done. In particular, for any subset of $\vertices(\compact{\pc})$ it must be checked 
whether or not 
the subset 
forms a face. However, we would like to find a solution avoiding the brute
force approach, as many of our examples are large and computationally
expensive.

\section{Hasse diagram of the compactification} %
\label{sec:face_lattice}

\subsection{Faces of the compactification of a single polyhedron}
In this section we will describe the faces of the compactification
$\compact{\poly}$ of a polyhedron $\poly$ with respect to its recession cone $\tail(\poly)$.
A node in $\hasse(\compact{\poly})$ corresponds to a face of \(\compact{\poly}\), so
we have to explain what these faces are. Looking at
\autoref{def:compactification}, we see that \(\compact{\poly}\) is made of 
polyhedra $\pi_{\sigma}(\poly)$  in the strata $N_{\RR}(\sigma)$. Notice that the set $\pi_{\sigma}(\poly)$ is 
 closed  in the stratum $N_{\RR}(\sigma)$, yet it is not closed in \(\compactifying(\Delta)\). 

For two cones $\tau\le\sigma\in\Delta=\tail(\poly)$ we get a map
\[
\pi_{\sigma,\tau}:\ \compactifying(\tau)\surj \compactifying(\sigma),
\]
such that $\pi_{\sigma}=\pi_{\sigma, \tau}\circ\pi_{\tau}$. Using this
definition, we can describe the compactification $\compact{F}$ in
$\compactifying(\Delta)$ of a face $F\le \pi_{\tau}(\poly)$ to be
\[
\compact{F}\ :=\ \Disjoint_{\sigma\in\Delta,\ \tau\le\sigma,  \ \relint(\pi_{\tau}(\sigma))\cap\tail(F)\not= \emptyset} \pi_{\sigma,\tau}(F).
\]
In particular, if $F$ is a compact face of $\pi_{\tau}(\poly)$, then
$\compact{F}=F$.

\begin{definition}[Face of $\compact{\poly}$]\label{def:compact_face}
   The \define{faces} of $\compact{\poly}$ are the compactifications $\compact{F}$ of
   any face $F\le \pi_{\tau}(\poly)$ of any $\pi_{\tau}(\poly)$.  For a face
   $\compact{F}$ of $\compact{\poly}$ that is the compactification of
   $F\le\pi_{\tau}(\poly)$ we call the cone $\trunk(\compact{F}):=\tau$ the
   \define{trunk} of $\compact{F}$. The set
   \[
   \support(\compact{F})\ :=\ \{\sigma\in\Delta\ |\ \tau\le\sigma, \relint(\pi_{\tau}(\sigma))\cap\tail(F)\not= \emptyset\}
   \]
are the \define{supporting cones} of $\compact{F}$.

Note that the trunk is the unique minimal element of $\support(\compact{F})$
and for $F$ being compact, it is the only element of $\support(\compact{F})$.
\end{definition}

\begin{example}
In \autoref{example:pos_orthant}, note that the face
\[
   F\ =\ (1,0) + \RR_{\ge 0}\cdot (1,0)\ =\
\begin{tikzpicture}[scale=.3, baseline=.2cm]
\draw[color=black!30] (-.3,-.3) grid (2.3,2.3);
\draw[thick, color=black!50] (0,2.3) -- (0,1) -- (1,0) -- (2.3,0);
\fill[pattern color=black!30, pattern=dots] (0,2.3) -- (0,1) -- (1,0) -- (2.3,0) -- (2.3,2.3) -- cycle;
\draw[color=black!30] (3,-.3) -- (3,2.3);
\draw[thick, color=black!50] (3,0) -- (3,2.3);
\draw[color=black!30] (-.3,3) -- (2.3,3);
\draw[thick, color=black!50] (0,3) -- (2.3,3);
\fill[color=black!50] (3,3) circle (2pt);
\fill[color=black!50] (1,0) circle (2pt);
\fill[color=black!50] (0,1) circle (2pt);
\fill[color=black!50] (3,0) circle (2pt);
\fill[color=black!50] (0,3) circle (2pt);
\fill (1,0) circle (6pt);
\draw[-stealth, line width=2pt] (1,0) -- (2.3,0);
\end{tikzpicture}
\]
of \(\poly=\pi_{0}(\poly)\) is not a face of $\compact{\poly}$.  However, its compactification, consisting
of $F$ and the vertex $\pi_{\RR_{\ge 0}\times\{0\}}(F)$, is a face of
$\compact{\poly}$.
\end{example}

Thus we can abbreviate the above formula for $\compact{F}$ as 
\[
\compact{F}\ :=\ \Disjoint_{\sigma\in\support(\compact{F})} \pi_{\sigma,\tau}(F),
\]
where $\tau=\trunk(\compact{F})$. In the following we will abbreviate this even
further by denoting the components $\pi_{\sigma,\tau}(F)$ as $F_{\sigma}$.

\begin{remark} \label{rem:simplified_support}
Since we are working in the case \(\Delta = \tail(\poly)\), we can reformulate the support condition. 
From \(F \le \pi_{\tau}(\poly)\), it follows that 
\(\tail(F) \le \tail(\pi_{\tau}(\poly))\) and by 
\(\tau \le \tail(\poly)\), we deduce 
\( \tail(\pi_{\tau}(\poly)) = \pi_{\tau}(\tail(\poly))\).
Clearly, for \(\tau \le \sigma \le \tail(\poly)\) also
\(\pi_{\tau}(\sigma) \le \pi_{\tau}(\tail(\poly))\). Thence we have
two faces of \(\pi_{\tau}(\tail(\poly))\) where the relative
interior of the first one intersects the second non-trivially. This
means that the first forms a face of the second.
So \[
\support(\compact{F})\ 
= \{\sigma\le\tail(\poly)\ |\ \tau\le\sigma, \pi_{\tau}(\sigma)\le\tail(F)\}.
\]
\end{remark}

\begin{example}
If $\poly$ and $\Delta$ are incompatible, \autoref{rem:simplified_support}
becomes invalid.

Consider \autoref{example:fan_incompatible}, and pick the face $F:=(0,0)$ of
$\poly$. Since $F$ is compact itself, we have $\compact{F}=F$. But if we
project along $\sigma:=\{0\}\times\RR_{\ge 0}\le\Delta$, the projection of
$\poly$ is the whole $\RR$ and the projection of $(0,0)$ cannot be a face,
since $\RR$ has no zero-dimensional faces.
\end{example}

\begin{example} 
Let us compute the $\support$ and $\trunk$ for some faces of the compactified
polyhedron in \autoref{example:pos_orthant}. Note that we will always write the
trunk as the first element in the support. 

The face $\compact{F}$ from \autoref{example:pos_orthant} has
$\trunk(\compact{F})=0$ and 
\[
   \support(\compact{F})\ =\
   \support\left(
      \begin{tikzpicture}[scale=.3, baseline=.2cm]
      \draw[color=black!30] (-.3,-.3) grid (2.3,2.3);
      \draw[thick, color=black!50] (0,2.3) -- (0,1) -- (1,0) -- (2.3,0);
      \fill[pattern color=black!30, pattern=dots] (0,2.3) -- (0,1) -- (1,0) -- (2.3,0) -- (2.3,2.3) -- cycle;
      \draw[color=black!30] (3,-.3) -- (3,2.3);
      \draw[thick, color=black!50] (3,0) -- (3,2.3);
      \draw[color=black!30] (-.3,3) -- (2.3,3);
      \draw[thick, color=black!50] (0,3) -- (2.3,3);
      \fill[color=black!50] (3,3) circle (2pt);
      \fill[color=black!50] (1,0) circle (2pt);
      \fill[color=black!50] (0,1) circle (2pt);
      \fill[color=black!50] (3,0) circle (2pt);
      \fill[color=black!50] (0,3) circle (2pt);
      \fill (3,0) circle (6pt);
      \fill (1,0) circle (6pt);
      \draw[-stealth, line width=2pt] (1,0) -- (2.3,0);
      \end{tikzpicture}
   \right)
   \ =\ 
   \{0,\ \RR_{\ge 0}\times\{0\}\}
   \ =\
   \left\{
      \begin{tikzpicture}[scale=.5, baseline=.2cm]
      \draw[color=black!30] (-.3,-.3) grid (1.3,1.3);
      \fill (0,0) circle (4pt);
      \end{tikzpicture},
      \begin{tikzpicture}[scale=.5, baseline=.2cm]
      \draw[color=black!30] (-.3,-.3) grid (1.3,1.3);
      \fill (0,0) circle (4pt);
      \draw[-stealth, line width=2pt] (0,0) -- (1.3,0);
      \end{tikzpicture}
   \right\}.
\]

Take $\compact{F'} = F' =\pi_{\RR_{\ge 0}^2}(\poly)$ in the previous example,
then
\[
   \support(\compact{F'})
   \ = \
   \support\left(
      \begin{tikzpicture}[scale=.3, baseline=.2cm]
      \draw[color=black!30] (-.3,-.3) grid (2.3,2.3);
      \draw[thick, color=black!50] (0,2.3) -- (0,1) -- (1,0) -- (2.3,0);
      \fill[pattern color=black!30, pattern=dots] (0,2.3) -- (0,1) -- (1,0) -- (2.3,0) -- (2.3,2.3) -- cycle;
      \draw[color=black!30] (3,-.3) -- (3,2.3);
      \draw[thick, color=black!50] (3,0) -- (3,2.3);
      \draw[color=black!30] (-.3,3) -- (2.3,3);
      \draw[thick, color=black!50] (0,3) -- (2.3,3);
      \fill[color=black!50] (3,3) circle (2pt);
      \fill[color=black!50] (1,0) circle (2pt);
      \fill[color=black!50] (0,1) circle (2pt);
      \fill[color=black!50] (3,0) circle (2pt);
      \fill[color=black!50] (0,3) circle (2pt);
      \fill (3,3) circle (6pt);
      \end{tikzpicture}
   \right)
   \ =\ 
   \{\RR_{\ge 0}^2\}
   \ = \
   \left\{
      \begin{tikzpicture}[scale=.5, baseline=.2cm]
      \draw[color=black!30] (-.3,-.3) grid (1.3,1.3);
      \fill[pattern color=black!50, pattern=dots] (1.3,0) -- (0,0) -- (0,1.3) -- (1.3,1.3) -- cycle;
      \fill (0,0) circle (4pt);
      \draw[-stealth, line width=2pt] (0,0) -- (1.3,0);
      \draw[-stealth, line width=2pt] (0,0) -- (0,1.3);
      \end{tikzpicture}
   \right\}.
\]

And for $\compact{F''}$ with $F''=\pi_{\RR_{\ge 0}\times\{0\}}(\poly)$, we get
\[
   \support(\compact{F''})
   \ = \
   \support\left(
      \begin{tikzpicture}[scale=.3, baseline=.2cm]
      \draw[color=black!30] (-.3,-.3) grid (2.3,2.3);
      \draw[thick, color=black!50] (0,2.3) -- (0,1) -- (1,0) -- (2.3,0);
      \fill[pattern color=black!30, pattern=dots] (0,2.3) -- (0,1) -- (1,0) -- (2.3,0) -- (2.3,2.3) -- cycle;
      \draw[color=black!30] (3,-.3) -- (3,2.3);
      \draw[thick, color=black!50] (3,0) -- (3,2.3);
      \draw[color=black!30] (-.3,3) -- (2.3,3);
      \draw[thick, color=black!50] (0,3) -- (2.3,3);
      \fill[color=black!50] (3,3) circle (2pt);
      \fill[color=black!50] (1,0) circle (2pt);
      \fill[color=black!50] (0,1) circle (2pt);
      \fill[color=black!50] (3,0) circle (2pt);
      \fill[color=black!50] (0,3) circle (2pt);
      \fill (3,3) circle (6pt);
      \fill (3,0) circle (6pt);
      \draw[line width=2pt, -stealth] (3,0) -- (3,2.3);
      \end{tikzpicture}
   \right)
   \ =\ 
   \{\RR_{\ge 0}\times\{0\},\ \RR_{\ge 0}^2\}
   \ = \
   \left\{
      \begin{tikzpicture}[scale=.5, baseline=.2cm]
      \draw[color=black!30] (-.3,-.3) grid (1.3,1.3);
      \fill (0,0) circle (4pt);
      \draw[-stealth, line width=2pt] (0,0) -- (1.3,0);
      \end{tikzpicture},
      \begin{tikzpicture}[scale=.5, baseline=.2cm]
      \draw[color=black!30] (-.3,-.3) grid (1.3,1.3);
      \fill[pattern color=black!50, pattern=dots] (1.3,0) -- (0,0) -- (0,1.3) -- (1.3,1.3) -- cycle;
      \fill (0,0) circle (4pt);
      \draw[-stealth, line width=2pt] (0,0) -- (1.3,0);
      \draw[-stealth, line width=2pt] (0,0) -- (0,1.3);
      \end{tikzpicture}
   \right\}.
\]
\end{example}

The following lemma will be crucial to show that faces in the
sense of \autoref{def:compact_face} behave as we expect from faces.
In particular, we will need it to guarantee that the faces of
\(\compact{\poly}\) form a polyhedral complex in the sense of 
\autoref{def:pc}.

\begin{lemma}\label{lemma:projection_of_face}
   Let $F\le \pi_{\tau}(\poly)$ be a face and let $\sigma\in\support(\compact{F})$.
   Then $\pi_{\sigma,\tau}(F)$ is a face of $\pi_{\sigma}(\poly)$.
\end{lemma}
\begin{proof}
   Since $\pi_{\sigma}=\pi_{\sigma, \tau}\circ\pi_{\tau}$, we may assume that
   $\tau = 0$ without loss of generality. Hence $\pi_{\tau}(\poly)=\poly$.
 
   Let $F\le \poly$ be a face and let $\sigma\in\support(\compact{F})$. 
   Thus, by \autoref{rem:simplified_support} the cone \(\sigma \le \tail(F)\).
   Since $F\le \poly$ there is a hyperplane
   $h\in \Hom(N_{\RR},\RR)$ such that
   \[
   F\ =\ \{p\in \poly\ |\ h(p) \mbox{ is minimal}\}.
   \]
   The hyperplane $h$ evaluates to a constant on $F$, hence the observation
   $\sigma\le\tail(F)$ implies $h(s)=0$ for all $s\in\sigma$. Thus, the hyperplane $h$
   is well-defined on $\compactifying(\sigma)$. This means that the set
   \[
   F'\ :=\ \{p\in\pi_{\sigma}(\poly)\ |\ h(p) \mbox{ is minimal}\}
   \]
   is a face of $\pi_{\sigma}(\poly)$. The observation $F'=\pi_{\sigma}(F)$ finishes
   the proof.
\end{proof}

\begin{example}\label{example:improper_Delta}
In \autoref{example:pos_orthant}, the projection $\pi_{\RR_{\ge 0}\times\{0\}}(F)$ is a vertex
of $\pi_{\RR_{\ge 0}\times\{0\}}(\poly)$. On the other hand, the projection
$\pi_{\{0\}\times\RR_{\ge 0}}(F)$ is not a face of $\pi_{\{0\}\times\RR_{\ge
0}}(\poly)$. In this case the support condition of the lemma is violated, and
$\{0\}\times\RR_{\ge 0}$ is not in the support of $\compact{F}$.

\end{example}

We will now use \autoref{lemma:projection_of_face} to show that faces in the
sense of \autoref{def:compact_face} form a polyhedral complex. We will start by
showing that it is compatible with taking faces.
\begin{lemma} \label{lemma:P_pc_1}
The face relation is transitive, i.e. 
if \(\compact{F'} \le\compact{F} \le \compact{\poly}\),
then \(\compact{F'} \le \compact{\poly}\).
\end{lemma}
\begin{proof}
By definition \(\compact{F} \le \compact{\poly}\) means that there exist
\(\tau \le \tail(\poly)\) and \(F \le \pi_{\tau}(\poly) 
\subset \compactifying(\tau) \cong \RR^{\text{codim}(\tau)}\).
When now considering \(\compact{F'} \le \compact{F}\) we want to
see \(F\) as a polyhedron in 
\(N'_{\RR} := \compactifying(\tau) \cong  \RR^{\text{codim}(\tau)}\)
and compactify with respect to the recession cone \(\tail(F)\).

Remember that 
\[
\support(\compact{F}) = \{ \sigma \le \tail(\poly) \ | \
\tau \le \sigma, \ \pi_{\tau}(\sigma) \le \tail(F)\}.
\]
The set \(\{ \pi_{\tau}(\sigma) \ | \ \sigma \in \support(\compact{F})\}\)
forms a fan, it is the fan given by \(\tail(F)\) and all its faces,
hence the fan with respect to which we compactify \(F\).

We have the face \(\compact{F'} \le \compact{F}\), thus by definition
\(F' \le \pi'_{\tau'}(F)\), where 
\(\pi'_{\tau'}: N'_{\RR} \to \compactifying'(\tau') \).
With the previous considerations \(\tau' = \pi_{\tau}(\sigma')\) 
for some \(\sigma' \in \support(\compact{F})\), and then
\(\pi'_{\tau'} = \pi_{\sigma',\tau}\).
So \(F' \le \pi_{\sigma',\tau}(F)\) and by \autoref{lemma:projection_of_face}
\(\pi_{\sigma',\tau}(F) \le \pi_{\sigma'}(\poly)\).
Thus \(F'\) also gives a face \(\compact{F'}\) of \(\compact{\poly}\).
We can see that the support of \(\compact{F'}\) in \(\compact{F}\) is the 
support of \(\compact{F'}\) in \(\compact{\poly}\) mapped with \(\pi_{\tau}\),
also the components of the compactification agree.
Thence, the compactification of \(F'\) is the same when
compactifying it as a face of \(\pi'_{\tau'}(F)\) 
or as face of \(\pi_{\sigma'}(\poly)\).
This yields the desired face relation \(\compact{F'} \le \compact{\poly}\).
\end{proof}

The following lemma shows that the intersection of two compact faces is again a
face in the sense of \autoref{def:compact_face}. 
\begin{lemma} \label{lemma:P_pc_2}
Let \(\compact{F}, \compact{F'} \le \compact{\poly}\) be two faces of \(\compact{\poly}\), 
then the intersection
\(\compact{F} \cap \compact{F'} \le \compact{\poly}\) is also a face.
\end{lemma}
\begin{proof}
We have \(F \le \pi_{\tau}(\poly)\) and \(F' \le \pi_{\tau'}(\poly)\).
Let us look at the intersection in a stratum \(\compactifying(\sigma)\):

\[
\begin{array}{rl}
\compact{F} \cap \compact{F'} \cap \compactifying(\sigma)
&= (\compact{F} \cap \compactifying(\sigma))
  \cap (\compact{F'} \cap \compactifying(\sigma))\\
&= \begin{cases}
 \pi_{\sigma,\tau}(F) \cap \pi_{\sigma,\tau'}(F') & 
 \sigma \in \support(\compact{F}) \cap \support(\compact{F'})\\
 \emptyset & \mbox{ else.}\\
 \end{cases}
 \end{array}
 \]

By \autoref{lemma:projection_of_face} we have the face relations
\(\pi_{\sigma,\tau}(F) \le \pi_{\sigma}(\poly)\)
and \(\pi_{\sigma,\tau'}(F') \le \pi_{\sigma}(\poly)\) and thus the intersection 
\[G_{\sigma}:=\pi_{\sigma,\tau}(F) \cap \pi_{\sigma,\tau'}(F') \le \pi_{\sigma}(\poly)\]
is a face of the polyhedron \(\pi_{\sigma}(\poly)\).

Our approach is to show that these \(G_\sigma\) form the components of a face of
\(\compact{\poly}\). First we construct a candidate for the support, in order
to find the \(\trunk:=\sigmamin\), from this we show that the \(G_\sigma\) form the
compactification of \(G_\sigmamin\le\pi_\sigmamin(\poly)\).

If \(\compact{F} \cap \compact{F'}\) is a face, its support should be
\[
\begin{array}{rl}
\support(\compact{F} \cap \compact{F'})
&=\support(\compact{F}) \cap \support(\compact{F'})\\ 
&=\left\{ \sigma \in \Delta \ \mid 
\begin{array}{l}
\tau \le \sigma,\ \pi_{\tau}(\sigma) \le \tail(F), \\
\tau' \le \sigma,\ \pi_{\tau'}(\sigma) \le \tail(F')
 \end{array} \right\}\\
\end{array}
 \]

If this set is non-empty, it contains a unique minimal element
\(\sigmamin\). Otherwise suppose that \(\sigma_1\not=\sigma_2\) are
both minimal elements of the set. Then both contain \(\tau\) as a
face, thus their intersection \(\sigma_1\cap\sigma_2\) does so,
too.  Now for the other condition, it holds for \(i =1,2\) that
\(\pi_{\tau}(\sigma_i) \le \tail(F)\). 
Hence, the intersection
\(\pi_{\tau}(\sigma_1)\cap\pi_{\tau}(\sigma_2) = \pi_{\tau}(\sigma_1 \cap \sigma_2)\) 
must be a face of \(\tail(F)\) as well.
The same applies
to \(\tau'\).  Thus \(\sigma_1 \cap \sigma_2\) is an element of the
set and strictly included in \(\sigma_i\), contradicting our
assumption.

Now set \(G := G_{\sigmamin} \le \pi_{\sigmamin}(\poly)\).
Then by definition \(\compact{G} \le \compact{\poly}\).
It remains to show that \(\compact{G} = \compact{F} \cap \compact{F'}\).
The concatenation law \(\pi_\sigma=\pi_{\sigma,\tau}\circ\pi_{\tau}\) extends
to \(\pi_{\sigma,\sigmamin}\circ\pi_{\sigmamin,\tau}=\pi_{\sigma,\tau}\) if
\(\tau\) is a face of \(\sigmamin\). One uses this to verify the equality
\(\compact{G} = \compact{F} \cap \compact{F'}\) on the non-trivial strata, i.e. those of
\(\support(\compact{F}) \cap \support(\compact{F'})\).
\end{proof}

\begin{remark}
In the proof of \autoref{lemma:P_pc_2} we use \(\Delta = \tail(\poly)\).
Otherwise the intersection \(\pi_{\tau}(\sigma_1 \cap \sigma_2)\) might not 
be a face of \(\tail(F)\).
\end{remark}

These lemmata ensure that by compactifying with respect to the recession cone,
we obtain a polyhedral complex in the sense of \autoref{def:pc}.

\begin{prop} \label{lemma:P_pc}
The compactification \(\compact{\poly}\) of a polyhedron \(\poly\)
inside the tropical toric variety of its recession cone 
\(\compactifying(\tail(\poly))\) is a polyhedral complex in
the sense of \autoref{def:pc}.
\end{prop}
\begin{proof}
The first point is elaborated in \autoref{lemma:P_pc_0}, 
the second point is \autoref{lemma:P_pc_1} 
and the third \autoref{lemma:P_pc_2}.
\end{proof}

Faces \(\compact{F}\) of the compactification \(\compact{\poly}\)
are closely related to faces of \(\poly\), in the following sense:
\begin{lemma}\label{lemma:preimage}
Let $F\le \pi_{\sigma}(\poly)$ be a face. Then the preimage
$\pi^{-1}_{\sigma}(F)\cap \poly$ is a face of $\poly$.
\end{lemma}

\begin{proof}
The face $F\le\pi_{\sigma}(\poly)$ is cut out by a hyperplane $h\in
\Hom(\compactifying(\sigma),\RR)$, i.e.
\[
F\ =\ \{p\in \pi_{\sigma}(\poly)\ |\ h(p) \mbox{ is minimal}\}.
\]
The preimage $\pi^{-1}_{\sigma}(F)\cap \poly$ is cut out by the composition
$h\circ\pi_{\sigma}$.
\end{proof}

In the following lemma we will show that each face \(\compact{F}\)
of the compactification \(\compact{\poly}\) is associated to a
unique face of \(\poly\). A key ingredient is
\autoref{lemma:projection_of_face} which ensures that the stratification is
compatible with the face structure.
\begin{lemma}\label{lemma:parent_independent}\label{def:parent}
Let $\compact{F}$ be a face of $\compact{\poly}$. Let
$\tau:=\trunk(\compact{F})$. Then for any $\sigma\in\support(\compact{F})$ we
have
\[
\pi_{\sigma}^{-1}(F_{\sigma})\cap \poly\ =\ \pi_{\tau}^{-1}(F_{\tau})\cap \poly.
\]
We call the face $\pi_{\tau}^{-1}(F_{\tau})\cap \poly$ the \define{parent face}
of $\compact{F}$, and denote this as \(\parent(\compact{F})\).
\end{lemma}
\begin{proof}
Assume for now that $\tau=0$. Then $\pi_{\tau}$ is just the identity and we
have to show that 
\[
F\ =\ (\pi_{\sigma}^{-1}\circ\pi_{\sigma}(F))\cap \poly
\]
for any cone $\sigma \in \support(\compact{F})$.
Just as in the proof of \autoref{lemma:projection_of_face}, the face $F$ is cut out from
$\poly$ by a hyperplane $h\in\Hom(N_{\RR},\RR)$ and because of
$\sigma\le\tail(F)$ this hyperplane is well-defined on $\compactifying(\sigma)$.
Thus, the hyperplane $h$ cuts out $\pi_{\sigma}(F)\le\pi_{\sigma}(\poly)$. Both $\pi_{\sigma}$ and
$\pi_{\sigma}^{-1}$ preserve the value of $h$, meaning that for any point $p\in
\compactifying(\sigma)$ we have $h(p)=h(q)$ for all points $q\in\pi_{\sigma}^{-1}(p)$
in the preimage. Denote by $h(F)$ the value of $h$ on $F$. Then
\[
(\pi_{\sigma}^{-1}\circ\pi_{\sigma}(F))\cap \poly\ =\ \{p\in \poly\ |\ h(p)=h(F)\}.
\]

The above argument also shows that
$\pi_{\sigma,\tau}^{-1}(F_{\sigma})\cap\pi_{\tau}(\poly)=F_{\tau}$.  Together with
the identity $\pi_{\sigma}=\pi_{\sigma,\tau}\circ\pi_{\tau}$ this finishes the
proof.
\end{proof}

\begin{remark}\label{cor:parent_equality}
Let \(\compact{F} \le \compact{\poly}\) with \(\trunk(\compact{F}) =\tau\). Then the following equality holds 
\[F = \pi_{\tau}(\parent(\compact{F})).\]
\end{remark}

\begin{example}
In \autoref{example:pos_orthant} the parent face of the vertex $\pi_{\RR_{\ge
0}\times\{0\}}((1,0)+\RR_{\ge 0}\cdot(1,0))$ is $(1,0)+\RR_{\ge 0}\cdot(1,0)$
itself.
\end{example}

\begin{remark}
With the notation of parent face, we can simplify the support condition for \(\compact{F}\), 
even further. From the statement in 
\autoref{rem:simplified_support} to
\[
\support(\compact{F})\ = \ 
\{\sigma \in \tail(\poly)\ | \ \trunk(\compact{F}) \le \sigma \le \tail(\parent(\compact{F}))\}.
\]
\end{remark}

Face relations between compact faces lift to a face relation of 
the parent faces.
\begin{lemma} \label{lemma:face_relation_to_parent}
Let $\compact{F'}\le\compact{F}$ be two faces of $\compact{\poly}$. Then the same
face relation holds for the parents, namely
\[
\parent(\compact{F'})\ \le\ \parent(\compact{F}).
\]
\end{lemma}
\begin{proof}
As in the proof of \autoref{lemma:P_pc_1} we have
\(F \le \pi_{\tau}(\poly)\) and 
\(F' \le \pi_{\sigma,\tau}(F) \le \pi_{\sigma}(\poly)\) 
for some \(\sigma \in \support(\compact{F})\).
Now \[\parent(\compact{F'}) = \pi_{\sigma}^{-1}(F') \cap \poly \le \poly\]
and 
\[\parent(\compact{F}) = \pi_{\tau}^{-1}(F) \cap \poly =\pi_{\sigma}^{-1}(F_{\sigma}) \cap \poly  \le \poly.\]

Since \(F' \subset F_{\sigma}\), also \(\parent(\compact{F'}) \subset \parent(\compact{F})\)
and because both are faces of \(\poly\), we get the required face relation of the parent faces.
\end{proof}

\begin{definition}\label{def:compact_face_dim}
Every face of $\compact{\poly}$ with trunk \(\sigma\) canonically inherits the dimension of the
underlying face of the $\pi_{\sigma}(\poly)$, i.e. for
$F\le\pi_{\sigma}(\poly)$ we set
\[
\dim(\compact{F})\ =\ \dim(F).
\]
\end{definition}
Topologically \autoref{def:compact_face_dim} makes sense. 
The dimension of \(\compact{F}\) should be the maximal length of a chain of faces of \(\compact{F}\).

Let \(\compact{F^0} < \compact{F^1} < \dots < \compact{F^d} = \compact{F}\) be a chain of faces 
of maximal length. Then we can consider the parent faces of the \(\compact{F^i}\).
By \autoref{lemma:face_relation_to_parent} this gives a chain of faces of \(\parent(\compact{F})\).
This also works when only taking the ``partial parent" \(\pi^{-1}_{\sigma,\tau}(F^i_{\sigma}) \cap F\),
and gives a chain of faces of \(F\).
On the other hand, the compactifications of a chain of faces of \(F\) gives a chain of faces of 
\(\compact{F}\).

After equipping faces of $\compact{\poly}$ with a dimension, it makes sense to
talk about vertices, i.e. faces of dimension zero.

\begin{prop}\label{prop:compact_vertices}
The vertices of $\compact{\poly}$ are the union of all the vertices of the
$\pi_{\sigma}(\poly)$.
\end{prop}
\begin{proof}
The main point is that all the vertices of the $\pi_{\sigma}(\poly)$ are already
compact. Since compactification preserves dimension, there cannot be more
vertices.
\end{proof}

\begin{prop}\label{prop:vertices}
The vertices of $\compact{\poly}$ are in one-to-one correspondence with faces $F$
of $\poly$ such that $\dim{F}=\dim{\tail(F})$.
\end{prop}
\begin{proof}
First assume we have $F\le \poly$ with $\dim{F}=\dim{\tail(F)}$. Now we choose
$\sigma=\tail(F)$. Then the projection $\pi_{\sigma}(F)$ is just a point. 
By \autoref{lemma:projection_of_face} it is a face of \(\pi_{\sigma}(\poly)\).

Conversely, assume we are given a vertex $v$ of $\compact{\poly}$. By
\autoref{prop:compact_vertices} we know it is the vertex of some
$\pi_{\sigma}(\poly)$. \autoref{lemma:preimage} gives that $F:=\parent(v)$ is a
face of $\poly$.
From \autoref{cor:parent_equality} we obtain \(\pi_{\sigma}(F) = v\).
Since \(\dim(v) =0\), it has to hold that \(\dim(F) = \dim(\tail(F))\).
\end{proof}

\begin{remark}
As in the previous proof and using \autoref{cor:parent_equality}
and \autoref{def:parent}, we can rewrite
\autoref{def:compact_face_dim} to
\[
\dim(\compact{F}) = \dim(F) = \dim(\parent(\compact{F}))-\dim(\trunk(\compact{F})).
\]
\end{remark}

\begin{example}
In \autoref{example:improper_Delta} consider the face $\RR_{\ge 0}\cdot(-1,1)$
of $\poly$. This does not give rise to a vertex.

If $\tail(\poly)$ is only refined by $\Delta$ as in
\autoref{condition:covering} of \autoref{def:compactification}, one face with $\dim{F}=\dim{\tail(F)}$ can give
rise to multiple vertices of $\compact{\poly}$ depending on how many
$\sigma\in\Delta$ there are with $\dim(\sigma)=\dim(\tail(F))$ and
$\sigma\subseteq\tail(F)$.

Take for example the following $\Delta$ in \autoref{example:pos_orthant}:
\[
\begin{array}{cc}
\Delta & \poly\mbox{ and }\compact{\poly}\\[.3cm]
\begin{tikzpicture}[scale=.7]
\draw[color=black!40] (-.3,-.3) grid (3.3,3.3);
\fill[pattern color=black!40, pattern=dots] (0,3.3) -- (0,0) -- (3.3,0) -- (3.3,3.3) -- cycle;
\draw[-stealth, thick] (0,0) -- (3.3,0);
\draw[-stealth, thick] (0,0) -- (3.3,3.3);
\draw[-stealth, thick] (0,0) -- (0,3.3);
\end{tikzpicture}\quad
&
\begin{tikzpicture}[scale=.7]
\draw[color=black!40] (-.3,-.3) grid (2.3,3.3);
\draw[color=black!40] (-.3,-.3) grid (3.3,2.3);
\draw[thick] (0,3.3) -- (0,1) -- (1,0) -- (3.3,0);
\draw[-stealth, thick] (0,1) -- (0,3.3);
\draw[-stealth, thick] (1,0) -- (3.3,0);
\fill[pattern color=black!40, pattern=dots] (0,2.3) -- (0,1) -- (1,0) -- (3.3,0) -- (3.3,2.3) -- cycle;
\fill[pattern color=black!40, pattern=dots] (0,3.3) -- (0,1) -- (1,0) -- (2.3,0) -- (2.3,3.3) -- cycle;
\draw[color=black!40] (4,-.3) -- (4,1.3);
\draw[-stealth, thick] (4,0) -- (4,1.3);
\draw[color=black!40] (-.3,4) -- (1.3,4);
\draw[-stealth, thick] (0,4) -- (1.3,4);
\draw[-stealth, thick] (3,3) -- (2.5,3.5);
\draw[-stealth, thick] (3,3) -- (3.5,2.5);
\fill (2,4) circle (2pt);
\fill (4,2) circle (2pt);
\fill (1,0) circle (2pt);
\fill (0,1) circle (2pt);
\fill (4,0) circle (2pt);
\fill (0,4) circle (2pt);
\end{tikzpicture}
\end{array}
\]
In this case the compactification has four new vertices instead of three.

One could refine the polyhedral complex, i.e. replace the polyhedron by a
polyhedral complex $\pc$ such that $\Delta=\tail(\pc)$ and
$\support(\pc)=\poly$. But in this case the compactification would end up
having five vertices. Hence, refinement does not serve as a trick to apply
our implementation for $\Delta\not=\tail(\pc)$.
\end{example}

\subsection{Compatibility of the single compactifications for polyhedral complexes}\label{sec:comp_pc}

In this section we prove that the compactification
$\compact{\pc}$ is an abstract polyhedral space. In the particular case when the polyhedral complex $\pc$ has a recession fan, we show that 
$\compact{\pc}$  forms a polyhedral complex  in the tropical toric variety $\compactifying(\tail(\pc))$. 

For the applications we have in mind, it suffices to consider polyhedral complexes
whose recession cones form a fan \(\Delta = \tail(\pc)\). 
In the sense of \autoref{def:fan_compatible}, \(\Delta\) will automatically be 
compatible and compactifying for \(\pc\).
Nevertheless, we give an example for a polyhedral complex \(\pc\) that has no 
recession fan. 

\begin{example} \label{example:no_rec_fan}
Take the following polyhedral complex in $\RR^3$:
\[
\begin{array}{ccc}
P_0 & := & (0,0,1) + \RR_{\ge 0} \cdot (1,1,0)\\
P_1 & := & \conv\{(0,0,0),(0,0,1)\}\\
P_2 & := & (0,0,0) + \RR_{\ge 0} \cdot (1,0,0) + \RR_{\ge 0} \cdot (0,1,0)
\end{array}.
\]
Here $\tail(P_0)\cap\tail(P_2)=\tail(P_0)$, but $\tail(P_0)$ is not a face of
$\tail(P_2)$, hence this polyhedral complex does not have a recession fan.
\end{example}

By \cite[Remark~3.13]{ossermanRabinoff} any polyhedral complex \(\pc\) has a compactifying 
fan which can be obtained by choosing a fan that is a refinement 
of the set of cones \(\cup_{\poly \in \pc} \tail(\poly)\).

\begin{example}
Any bounded polyhedral complex has a recession fan, consisting just of the
origin. Nevertheless, since these are already compact, they are not interesting
for our procedures.
Another trivial example for polyhedral complexes that always have a
recession fan, are those that consist of one polyhedron and all its
faces.
\end{example}

First we inspect the first condition for the compactification to form a polyhedral complex.
The following lemma extends  \cite[Corollary~3.7]{ossermanRabinoff} for a polyhedron to polyhedral complexes. 

\begin{lemma}\label{lemma:pc_in_stratum}
If $\pc$ has a recession fan $\Delta$, then the intersection of the
compactification of \(\pc\) with respect to \(\Delta\) with a
stratum \(\compact{\pc} \cap \compactifying(\rho)\) forms a
polyhedral complex in \(\RR^{\text{codim}{\rho}}\).
\end{lemma}
\begin{proof}
The intersection with the stratum \(\compactifying(\rho)\) consists of the following polyhedra
\[
\begin{array}{rl}
\compact{\pc} \cap \compactifying(\rho) 
&= \{\compact{\poly} \cap \compactifying(\rho) \ | \ \poly \in \pc\}\\
&= \{\pi_{\rho}(\poly) \ | \ \poly \in \pc, \rho \in \support(\compact{\poly})\}
\end{array}.
\]
A cone \(\rho\) is contained in \(\support(\compact{\poly})\) if and only if \(\relint(\rho) \cap \tail(\poly) \not= \emptyset\).
Since \(\Delta=\tail(\pc)\) is the recession fan and \(\rho, \tail(\poly) \in \Delta\) this condition is equivalent to
\(\rho \le \tail(\poly)\), thus
\[
\begin{array}{rl}
\compact{\pc} \cap \compactifying(\rho) &= 
\{\pi_{\rho}(\poly) \ | \ \poly \in \pc, \rho \le \tail(\poly)\}.\\
\end{array}
\]
First we show that the first axiom for being a polyhedral complex
holds for this collection of polyhedra, namely that if \(F \le
\pi_{\rho}(\poly)\), then \(F\) comes from an element of \(\pc\).
This element is \(\parent(\compact{F})\), which is a member of \(\pc\) by
\autoref{lemma:preimage}, and  with \autoref{cor:parent_equality} 
\(F = \pi_{\rho}(\parent(\compact{F}))\). 

For the second axiom of a polyhedral complex, let \(\pi_{\rho}(\poly),\pi_{\rho}(\poly') \in \compact{\pc} \cap \compactifying(\rho)\).
We want to show that the intersection 
\(\pi_{\rho}(\poly) \cap \pi_{\rho}(\poly') \in \compact{\pc} \cap \compactifying(\rho)\).
But \(\pi_{\rho}(\poly) \cap \pi_{\rho}(\poly') = \pi_{\rho}(\poly \cap \poly')\).
Since \(\pc\) is a polyhedral complex \(\poly \cap \poly' \in \pc\) and 
by \(\rho \le \tail(\poly), \tail(\poly')\), also \(\rho \le \tail(\poly \cap \poly')\).
Thence, the intersection \(\pi_{\rho}(\poly) \cap \pi_{\rho}(\poly') \in \compact{\pc} \cap \compactifying(\rho)\).
\end{proof}

In contrast to \cite[Lemma~3.10/Proposition~3.12]{ossermanRabinoff}, which
analyses  the support of \(\compact{\pc}\) in \(\compactifying(\Delta)\), the
following theorem captures the combinatorial structure of a polyhedral complex
on \(\compact{\pc}\).

\begin{theorem}\label{prop:pc}
If $\pc$ has a recession fan $\Delta$, then the compactification
$\compact{\pc}$ in \(\compactifying(\Delta)\) forms a polyhedral complex.
\end{theorem}
\begin{proof}
For \(\compact{\pc}\) to form a polyhedral complex, we have to check the three points from \autoref{def:pc}.

\begin{asparaenum}
\item This is \autoref{lemma:pc_in_stratum}.
\item The proof of \autoref{lemma:P_pc_1} can be used here. 
The fact that we glue does not affect this condition.

\item Let \(Q=\compact{F} \le \compact{\poly}\) and  
\(Q'=\compact{F'} \le \compact{\poly'}\). 
If \(\poly = \poly'\), then \(Q \cap Q' \in \compact{\poly} \subset
\compact{\pc}\) by \autoref{lemma:P_pc_2}.
Let us investigate the case \(\poly \neq \poly'\). 
Without loss of generality \(F=\pi_{\tau}(\poly)\) and \(F'=\pi_{\tau'}(\poly')\). 
(Just choose \(\poly = \parent(\compact{F})\).)
And 
\[\support(\compact{F})
 =\support(\compact{\pi_{\tau}(\poly)})
= \{\sigma \in \tail(\pc) \ | \ \tau \le \sigma \le \tail(\poly)\}.\]
Thence, 
\[
\begin{array}{rl}
\support(\compact{F}) \cap \support(\compact{F'}) &=
 \left\{ \sigma \in \tail(\pc) \ | \ 
\begin{array}{l}
\tau \le \sigma \le \tail(\poly)\\ 
\tau' \le \sigma \le \tail(\poly')\\ 
\end{array}
\right\}\\
& = \left\{ \sigma \in \tail(\pc) \ | \ 
\langle \tau + \tau' \rangle \le \sigma \le \tail(\poly \cap \poly')
\right\}
\end{array}
\]
As already elaborated in the proof of \autoref{lemma:P_pc_2}
this set has a minimal element \(t\).  

The intersection of \(\compact{F} \cap \compact{F'}\)  with a
stratum \(\compactifying(\sigma)\) for 
\(\sigma \in \support(\compact{F}) \cap \support(\compact{F'})\) is
\[ 
\begin{array} {rl}
\compact{F} \cap \compact{F'} 
\cap \compactifying(\sigma) &=
\pi_{\sigma,\tau}(F) \cap \pi_{\sigma, \tau'}(F)\\ 
&= \pi_{\sigma}(\poly) \cap \pi_{\sigma}(\poly')\\ 
&= \pi_{\sigma}(\poly \cap \poly') 
\end{array},
\]
otherwise it is empty.
Then 
\(\compact{F} \cap \compact{F'} = \compact{\pi_t(\poly \cap \poly')}\). 
This is a face of \(\compact{\poly \cap \poly'} \le \compact{\poly}\) and by
the previous point in \(\compact{\pc}\).
It is also a face of \(\compact{\poly'}\). Since we glue along faces,
we have exactly one element in \(\compact{\pc}\).
\end{asparaenum}
\end{proof}

\begin{example}\label{example:hypersurface}
A tropical hypersurface $X_f$ in $\mathbb{R}^{n}$ is defined by a tropical
polynomial $f$ which is a convex piecewise integral affine function. The Newton
polytope $NP_f$ is the support of the hypersurface $X_f$. The hypersurface
is also equipped with weights on its top dimensional faces \cite{tropIntro}.
The collection of recession cones of the faces of $X_f$ is a fan $\Delta_f$. In
fact, it  is the codimension one skeleton of the dual fan of the Newton
polytope of $f$. 
 
 The compactification of $X_f$ in $N_{\mathbb{R}}(\Delta_f)$ is  a stratified
 space. The strata are in correspondence with the faces $\{F\}$ of $NP(f)$ and
 each stratum   is a tropical hypersurface coming from the restriction $f|_{F}$
 of the tropical polynomial to the monomials corresponding to lattice points
 contained in $F$. 
\end{example}

\begin{example} \label{example:matroid}
If $M$ is a matroid on the ground set $\{0, \dots, n\}$  of rank $d+1$, the
matroidal fan $\Sigma(M)$ of $M$ is a simplicial fan  in
$\mathbb{R}^{n+1}/\langle (1, \dots, 1)\rangle$ which is isomorphic to the cone
over the order complex  of the lattice of flats of $M$ \cite{ArdilaKlivans}.
The fan $\Sigma(M)$ also defines a tropical toric variety $TV_\mathbb{T}({\Sigma(M)}).$
Compactifying the fan face by face or taking its closure in
$TV_\mathbb{T}({\Sigma(M)})$ yield the same complex $\overline{\Sigma(M)}$.
Since the fan is simplicial, the faces of the compactification are all cubes.
In general, if the fan is simplicial, one obtains the so called \define{canonical
compactification} of \cite[1.4]{amini_cubical}.
The tropical homology \cite{ims:tropicalBook} of the  fan $\overline{\Sigma(M)}$ is isomorphic to the
Chow ring of a matroid of \cite{FeichtnerSturmfels} by \cite{amini_cubical}.

There are other simplicial fan structures on the set $\text{Supp}(\Sigma(M))$ coming from building sets \cite{FeichtnerSturmfels}. Distinct fan structures give distinct compactifications and the tropical homology of these compactifications are also distinct. However for  a fan structure coming from a building set $\mathcal{G}$ we have $ H^q(\Sigma(M, \mathcal{G}); \mathcal{F}^p) = 0 $ if $p \neq q $ and 
$A^{2k}(M, \mathcal{G}) \cong H^k(\Sigma(M, \mathcal{G}); \mathcal{F}^k)$ otherwise. 

Tropical linear spaces  are polyhedral complexes in Euclidean space coming from valuated matroids \cite{Speyer}. 
The recession fan of a tropical linear space is supported on the fan of the underlying matroid of the valuated matroid. Therefore, for a suitable fan structure $\Sigma(M)$, the compactification of a tropical linear space
is a polyhedral complex in $TV_\mathbb{T}({\Sigma(M)})$.

\end{example}

If $\pc$ does not have a recession fan, we can describe a compactification
$\compact{\pc}$ by compactifying every polyhedron in its recession cone. The
resulting object does not live in a tropical toric variety,  nevertheless it is a polyhedral space.

\begin{example} \label{example:no_tail_fan:refinement}
Let \(\pc\) be the polyhedral complex from \autoref{example:no_rec_fan}.
It does not have a recession fan, but the following \(\Delta\) is a compactifying fan for \(\pc\).
The fan is given by  
\(\Delta = \{0,\ \rho_0,\ \rho_1,\ \rho_2, \sigma_1,\ \sigma_2\}\) where
\(\rho_0 = \RR_{\ge 0} (1,\ 0,\ 0)\), \(\rho_1 = \RR_{\ge 0}(1,\ 1,\ 0)\), 
\(\rho_2=\RR_{\ge 0}(0,\ 1,\ 0)\), \(\sigma_1 = \RR_{\ge 0} \rho_0 + \RR_{\ge 0} \rho_1\)
and \(\sigma_2 = \RR_{\ge 0} \rho_1 + \RR_{\ge 0} \rho_2\).

Then \(\support(\poly_0) = \{0, \rho_1\}\), \(\support(\poly_1) = \{0\}\) 
and \(\support(\poly_2) = \Delta\).
For the lower-dimensional faces of \(\poly_i\) there are only two interesting ones:
\(\rho_0 \le \poly_2\) with support \(\{0,\rho_2\}\) and \(\rho_2 \le \poly_2\)
with support \(\{0,\rho_2\}\).

Consider the intersection with the stratum \(\compactifying(\rho_1)\):
\[\compact{\pc} \cap \compactifying(\rho_1) = \{\pi_{\rho_1}(\poly_0), \pi_{\rho_1}(\poly_2)\}\]
The projection \(\pi_{\rho_1}(\poly_0) = (0,\ 1)\) is just a point,
the projection \(\pi_{\rho_1}(\poly_2) = \RR (1,0)\) is a line.
These two polyhedra form a non-connected polyhedral complex.

If we subdivided \(\pc\) into \(\pc'\) by subdividing \(\poly_2\) such that \(\tail(\pc') = \Delta\),
we also obtain a non-connected polyhedral complex in this stratum, but the 
line will be subdivided into two cones with common vertex.
\end{example}

\begin{example}\label{example:no_tail_fan:abstract}
Instead of refining the fan, we could also apply our algorithm directly to the
individual polyhedra of \autoref{example:no_rec_fan}. We visualise this with the
following two pictures, where the right hand picture contains the new faces of
the compactification. Coordinates at arrow tips indicate their direction.
\[
\begin{tikzpicture}[x={(1cm,0cm)},y={(-.2cm,-.3cm)},z={(0cm,1cm)}, scale=2]
\fill[pattern color=black!60, pattern=dots] (0,0,0) -- (1,0,0) -- (1,1,0) -- (0,1,0) -- cycle;
\draw[thick, -stealth] (0,0,0) -- (1.3,0,0) node[right] {$(1,0,0)$};
\draw[thick, -stealth] (0,0,0) -- (0,1.3,0) node[below] {$(0,1,0)$};
\draw[thick] (0,0,0) node[left]{$(0,0,0)$} -- (0,0,1) node[left] {$(0,0,1)$};
\draw[thick, -stealth] (0,0,1) -- (1.3,1.3,1) node[below right]{$(1,1,0)$};
\end{tikzpicture}
\quad
\begin{tikzpicture}[x={(1cm,0cm)},y={(-.2cm,-.3cm)},z={(0cm,1cm)}, scale=2]
\fill[pattern color=black!60, pattern=dots] (0,0,0) -- (1,0,0) -- (1,1,0) -- (0,1,0) -- cycle;
\draw[thick, -stealth] (0,0,0) -- (1.3,0,0);
\draw[thick, -stealth] (0,0,0) -- (0,1.3,0);
\draw[thick] (0,0,0) -- (0,0,1);
\draw[thick, -stealth] (0,0,1) -- (1.3,1.3,1);
\def\x{.2}
\fill (1.3+\x,1.3+\x,1) circle (2pt);
\fill (1.3+\x,1.3+\x,0) circle (2pt);
\fill (1.3+\x,0,0) circle (2pt);
\fill (0,1.3+\x,0) circle (2pt);
\draw[line width=.1cm] (1.3+\x,1.3+\x,0) -- (1.3+\x,0,0);
\draw[line width=.1cm] (1.3+\x,1.3+\x,0) -- (0, 1.3+\x,0);
\end{tikzpicture}
\]
Note that two polyhedra $P$ and $Q$ whose recession cones intersect improperly
must have $P\cap Q= \emptyset$. Thus checking the face relation on this
intersection is trivial.  However, the compactification $\compact{\pc}$ lacks a
canonical embedding into a tropical toric variety.
\end{example}

\begin{theorem}\label{prop:polyspace}
The   compactification
$\compact{\pc}$ is a compact polyhedral space.%
\end{theorem}

\begin{proof}
The closure of each face $\compact{P}$ is a topological space, as it can be equipped with the subspace topology from its inclusion in $N_{\RR}(\tail(P))$. Moreover, it is clearly second countable. 
We can specify a topology on $\compact{\pc}$ by insisting that the pullbacks of all inclusions $\compact{P} \hookrightarrow \compact{\pc}$ be continuous. This  topology on $\compact{\pc}$ is then also second countable since 
$\pc$ is has a finite number of faces. %
The space $\compact{\pc}$ is compact. Moreover, distinct points can be separated by open neighborhoods, so it is Hausdorff. %

Equipping $\compact{\pc}$  with the collection of charts $\{ \phi_{\tau}: U_{\tau} \to \mathbb{T}^{r_{\tau}} \times \RR^{n_{\tau}}\}$, where $\tau$ is a face of $\pc$ and $U_{\tau}$ is the open star of $\tau$ in $\compact{\pc}$, makes $\compact{\pc}$ an abstract polyhedral space. 
\end{proof}

\section{Data structure for the compactification}\label{sec:datastr}
In this section we will move to the more algorithmic part, and describe how one
can encode the face structure of the compactification algorithmically.

Given a polyhedron $\poly$, we use the following conventions:
\begin{itemize}
\item $V=\{v^0,\ldots, v^m\}$ denotes the set of vertices of the polyhedron $\poly$,
\item $R=\{\rho^0,\ldots,\rho^n\}$ denotes the set of rays of $\poly$.
\end{itemize}

\begin{definition}\label{def:realisation}
Denote by $\newVert=\vertices(\compact{\poly})$ the vertices of the
compactification of $\poly$. Then every vertex $a\in\newVert$ comes with a face
$F_a:=\parent(a)$ of $\poly$ such that $\dim(F_a)=\dim(\tail(F_a))$. So $F_a$ can
be uniquely represented via the elements of $V\disjoint R$ and we call this
representation the \define{realisation} $\realisation(a)$ of $a$.

Furthermore we define two maps $\sedentarity$ and $\nu$:
\[
\sedentarity(a)\ :=\ R\cap\realisation(a)\ \mbox{ and }\ \nu(a)\ :=\ V\cap\realisation(a).
\]
\end{definition}

\begin{definition}
For a subset $S\subseteq\newVert$ we define
\begin{itemize}
\item $\realisation(S):=\cup_{a\in S}\realisation(a)$ the \define{original realisation} of $S$,
\item $\minface{S}$ the \define{smallest face of $\poly$ containing $\realisation(S)$}.
\item $\minfacevert{S} := \vertices(\minface{S})\subseteq V\disjoint R$ the vertices and rays of $\minface{S}$.
\item $\sedentarity(S):=\cap_{a\in S}\sedentarity(a)$ the \define{sedentarity} of $S$,
\end{itemize}
\end{definition}

Every face of the polyhedron $\poly$ has a unique description in terms of
elements of $V\disjoint R$. The compactified polyhedron $\compact{\poly}$
has no rays, since it is compact. We already know the vertices of
$\compact{\poly}$ by \autoref{prop:compact_vertices}, hence analogously every
face of $\compact{\poly}$ has a unique description as a subset of $\newVert$.

The following lemma connects the maps $\pi_{\sigma}$ with the realisation map
$\realisation$.
\begin{lemma}\label{lemma:minface_is_parent}
For a face $\compact{F}$ of $\compact{\poly}$, we have
\[
\minface(\vertices(\compact{F}))\ =\ \parent(\compact{F}).
\]
\end{lemma}
\begin{proof}
For the inclusion $\minface(\vertices(\compact{F}))\subseteq
\parent(\compact{F})$, pick any vertex $a\in\vertices(\compact{F})$. By
\autoref{lemma:face_relation_to_parent} we have $\parent(a) \le
\parent(\compact{F})$. So $\parent(a)$ is a face of $\parent(\compact{F})$ for
all vertices $a\in\vertices(\compact{F})$. Then $\parent(\compact{F})$ must
contain the minimal face of $\poly$ containing all the $\parent(a)$.

For other inclusion $\minface(\vertices(\compact{F}))\supseteq \parent(\compact{F})$,
assume that
\[
F':=\minface(\vertices(\compact{F}))\subsetneq\parent(\compact{F}).
\]
Then
$F'<\parent(\compact{F})$ is a face of $\parent(\compact{F})$. Our goal is to
arrive at a contradiction. Let $\tau=\trunk(\compact{F})$, 
then \(\tau \le \tail(F')\), and we can consider the
face $\compact{\pi_{\tau}(F')}$ of $\compact{\poly}$. 
We will show that
$\vertices(\compact{F})\subseteq\vertices(\compact{\pi_{\tau}(F')})$. This
implies that $\compact{F}$ is a face of $\compact{\pi_{\tau}(F')}$, and hence
again by \autoref{lemma:face_relation_to_parent}
\[
\parent(\compact{F})\le\parent(\compact{\pi_{\tau}(F')}).
\]
But as we have seen in the proof of \autoref{lemma:parent_independent}
for \(\tau \in \support(\compact{F'})\) it holds that 
\(\parent(\compact{\pi_{\tau}(F')})=F'\) which contradicts our initial assumption.

Now take any vertex $a\in\vertices(\compact{F})$. Since
\[
\realisation(a)\subseteq\realisation(\vertices(\compact{F}))\subseteq\vertices(F'),
\]
we know that $\parent(a)$ is a face of $F'$, and hence
$\tau\le\tail(\parent(a))\le\tail(F')$. Thus
$\tail(\parent(a))\in\support(\compact{\pi_{\tau}(F')})$, implying
$a\in\vertices(\compact{\pi_{\tau}(F')})$.
\end{proof}

\begin{remark}\label{remark:minface}
Taking $\minface$ on the left hand side in
\autoref{lemma:minface_is_parent} is necessary, i.e. in general we only have
\[
\realisation(\vertices(\compact{F}))\ \subsetneq\ \vertices(\parent(\compact{F})).
\]
Consider the face $\compact{F}$ of $\compact{\poly}$ as in the figure below in
$\mathbb{R}^2$ with recession cone generated by the direction $(1,0)$. 
\[
\begin{tikzpicture}
\draw[color=black!40] (-1.3,-1.3) grid (4.3,1.3);
\fill[pattern color=black!40, pattern=dots]
(4.3,1) 
-- (0,1) node[above] {$v_0$} 
-- (-1,0) node[left] {$v_2$} 
-- (0,-1) node[below] {$v_1$} 
-- (4.3,-1) -- cycle;
\draw[thick] (4,1) -- (0,1) -- (-1,0) -- (0,-1) -- (4,-1);
\draw[-stealth, thick] (0,1) -- (4.3,1) node[right] {$\rho_0$};
\draw[-stealth, thick] (0,-1) -- (4.3,-1) node[right] {$\rho_1$};
\draw[thick] (6,-1) -- (6,1);
\node at (0,0) {\textbf{$\poly$}};
\fill (6,1) node[above] {$a_0$} circle (2pt);
\fill (6,-1) node[below] {$a_1$} circle (2pt);
\node[right] at (6,0)  {$\compact{F}$};
\end{tikzpicture}
\]
The compactification has a face $\compact{F}$ at infinity  that is a
line segment. Its preimage is the whole polyhedron $\poly$, but the realisations of the
vertices do not contain the middle vertex that is adjacent to two bounded edges.
Since $\realisation(\vertices(\compact{F})) =   \realisation(a_0) \cup  \realisation(a_1) = \{v_0, v_1, \rho_0, \rho_1\},$
but  $\vertices(\parent(\compact{F})) =   \{v_0, v_1, v_2, \rho_0, \rho_1\}.$

Note that for a vertex \(a\) it holds that \(\minfacevert(a) =
\realisation(a)\), this is due to \autoref{prop:vertices}.
\end{remark}

\begin{remark}
The image in \autoref{remark:minface} also highlights an important difference
of the tropical compactification versus the compactification in tropical
projective space as described in \cite[Fig. 16]{digraph}. Here parallel lines
get different end points, while in \cite{digraph} they would get the same.
\end{remark}

\begin{lemma}\label{lemma:trunk_from_rays}
For a face $\compact{F}$ of $\compact{\poly}$, we have
\[
\cone(\sedentarity(\vertices(\compact{F})))\ =\ \trunk(\compact{F}).
\]
\end{lemma}
\begin{proof}
Pick a vertex $a\in\vertices(\compact{F})$. Then this is a vertex of some
$F_{\sigma}$. Now observe that all vertices $a\in\vertices(F_{\sigma})$ have
$\tail(\parent(a))=\sigma$. Since $\tau = \trunk(\compact{F})$ is the unique
minimal element of $\support(\compact{F})$, we only need to make sure that
$F_{\tau}$ has a vertex. 
This is true, since we compactify with respect to the recession cone $\tail(\poly)$.
\end{proof}

Given a subset $S\subseteq \newVert$ we want to determine whether it is the set of
vertices of a face of the compactification of $\poly$. This can be done using the closure operator. 

\begin{theorem}\label{prop:closure_operator}
Define the set $\compact{S}\ :=\ \{a\in \newVert\ |\ \sedentarity(S)\subseteq \realisation(a)\subseteq \minfacevert(S)\}$.
The set $S$ is the vertex set of a face of the compactification of $\poly$ if
and only if $S=\compact{S}$. The set $\compact{S}$ is the smallest face of the
compactification $\compact{\poly}$ containing $S$, and hence, the operator
$S\mapsto \compact{S}$ is a closure operator as in \autoref{def:closure_operator}.
\end{theorem}
\begin{proof}
First we will prove the implication ``$\Rightarrow$''.
Let $\compact{F}$ be a face of $\compact{\poly}$. Define
$S:=\vertices{\compact{F}}$. We want to show that $S=\compact{S}$. The
inclusion $S\subseteq\compact{S}$ is trivial. 
Let $a\in\compact{S}$. 
Then
$\realisation(a)\subseteq\minfacevert(S)$ implies that $\parent(a)$ is contained in
$\parent(\compact{F})$, in particular it is a face. Thus, the recession cone $\tail(\parent(a))$ is a
face of $\tail(\parent(\compact{F}))$. 
The condition
$\sedentarity(S)\subseteq\realisation(a)$ together with \autoref{lemma:trunk_from_rays}
implies that $\trunk(\compact{F})$ is contained in $\tail(\parent(a))$. 
Thus $\tail(\parent(a))\in\support(\compact{F})$ and $a$ is
a vertex of $F_{\tail(\parent(a))}$.

For the other direction we have $S=\compact{S}$ and want to show that $S$ is the vertex set of a face
$\compact{F}$. Thus we pick $\tau=\cone(\sedentarity(S))$. Furthermore pick
\(\mathcal{F} = \minface(S)\).
Then we claim that $S$ is the vertex set of
$\compact{F} := \compact{\pi_{\tau}(\mathcal{F})}$. 
Denote by
$S':=\vertices(\compact{F})$. Then by \autoref{lemma:trunk_from_rays} 
\[
\cone(\sedentarity(S'))\ =\ \trunk(\compact{F})\ =\ \tau\ =\ \cone(\sedentarity(S)).
\]
Furthermore we have
\[
\minfacevert(S')\ =\ \vertices(\parent(\compact{F}))\ =\ \vertices(\mathcal{F})\ =\ \minfacevert(S)
\]
due to \autoref{lemma:minface_is_parent} and \autoref{lemma:parent_independent}. Since
$\compact{\compact{S}}=\compact{S}$ and faces of $\poly$ are closed as well, we get
$S=S'$.
\end{proof}

\begin{theorem}
Let $\pc$ be a polyhedral complex in $N_{\mathbb{R}}$ that has recession fan
$\Delta = \tail(\pc)$. Then the Hasse diagram of the closure $\compact{\pc}$
in $N_{\mathbb{R}}(\Delta)$ is computed via Ganter's algorithm using
the closure operator defined in \autoref{prop:closure_operator}. Improperly
intersecting recession cones live in different charts of the polyhedral space.
\end{theorem} 

\section{Signed incidence relations on compactifications}
\label{sec:sir}
With other computational goals in mind, it is useful to equip the Hasse diagram 
of the compactification with  a signed incidence relation, also known as an orientation map.
Such a  map is required to compute (co)-homology of the compactification and of cellular (co)-sheaves on it. 
 Details on cellular (co-)sheaves can be found
in \cite{curry_thesis} and \cite{ourTropHom}. %

\begin{definition}[{\cite[Definition~6.1.9]{curry_thesis}}]
Given a polyhedral complex $\pc$, a \define{signed incidence relation} is a map
\[
\begin{array}{ccc}
   \pc\times\pc & \to & \{0,\pm 1\}\\
   (\sigma,\tau) & \mapsto & [\sigma,\tau],
\end{array}
\]
such that
\begin{itemize}
   \item If $[\sigma,\tau]\not= 0$ then $\sigma\le \tau$; and
   \item For any pair $\sigma,\tau\in\pc$ we have
      $\sum_{\gamma}[\sigma,\gamma][\gamma,\tau]=0$.
\end{itemize}
\end{definition}
The original definition is for general cell complexes. We will rephrase this
definition for our concrete setting. %
\begin{definition}\label{def:signedIR}
A \define{signed incidence relation} on a polyhedral complex $\pc$ is a map
\[
\signedIR:\ \edges(\hasse(\pc))\quad \to \quad \{\pm 1\}
\]
such that for any two nodes $u,w\in\nodes(\hasse(\pc))$ we have
\[
\sum_{v} \signedIR(u,v)\signedIR(v,w) = 0.
\]
\end{definition}

\begin{lemma}[{\cite[Lemma~6.1.8]{curry_thesis}}]
For a polyhedral complex $\pc$ any non-trivial equation of
\autoref{def:signedIR} looks like
\[
\signedIR(u,v)\signedIR(v,w) + \signedIR(u,v')\signedIR(v',w)\ =\ 0.
\]
\end{lemma}
This lemma means that we just have to solve equations for squares in the Hasse
diagram to arrive at a valid signed incidence relation.

\begin{algorithm}
\begin{algorithmic}[1]
\Procedure{sir}{$\hasse(\pc)$}
\State Initialize $\signedIR$ to be zero everywhere.
\For {All edges $\emptyset\to u$ in $\edges(\hasse(\pc))$}
   \State $\signedIR(\emptyset, u)\gets 1$
\EndFor
\For {$i=2,\ldots,\dim(\pc)$}
   \For {$u\in\nodes_i(\hasse(\pc))$}
      \State $squares \gets \{[v,v',w]\in\nodes(\hasse(\pc))^3\ |\ \dim(v)=\dim(v')=i-1,\ \dim(w)=i-2,\ (w,v),(w,v'),(v,u),(v',u)\in\edges(\hasse(\pc))\}$
      \State $[v_0,v_0',w_0] \gets squares[0]$ \label{step:random}
      \State $\signedIR(v_0,u) \gets 1$ \label{step:initial}
      \State Solve all squares of $squares$
   \EndFor
\EndFor
\State \Return $\signedIR$
\EndProcedure
\end{algorithmic}
\caption{Signed incidence relation}
\label{alg:signedIR}
\end{algorithm}

\begin{prop}
\autoref{alg:signedIR}  produces a signed incidence relation on
the Hasse diagram $\hasse(\pc)$ of a polyhedral complex, as well as on the
Hasse diagram $\hasse(\compact{\pc})$ of its compactification.
\end{prop}
\begin{proof}
From \cite{curry_thesis} we know that we can get a signed incidence relation on $P$ by
choosing a basis for the affine hull of every face. For an edge $(u,v)$ in the
Hasse diagram we assign $1$ if the orientations of the respective bases agree,
and $-1$ otherwise. \autoref{alg:signedIR} omits the step of choosing a basis. Instead
it chooses a random edge $(v_0,u)$ in Step~\ref{step:initial}, and assigns $1$
as its signed incidence relation. Assuming the signed incidence relation is
known for all edges whose endpoint has dimension $<\dim{u}$, the signed
incidence relation is now uniquely determined for any edge ending in $u$. We
apply this procedure for any node of dimension $\dim{u}$ and then proceed
inductively over the dimension.
\end{proof}

\section{Implementation in \polymake}\label{sec:polymake}
The closure operator of \autoref{prop:closure_operator} can now be plugged into
Ganter's algorithm of \autoref{sec:ganter}. In \polymake the datatype of the
Hasse diagram is a directed graph, with a decoration giving
auxiliary information for every node. This auxiliary information contains the
indices of the vertices forming the associated face and the dimension of the
face. The only missing information to determine the vertices of $\compact{\pc}$
as described in \autoref{prop:vertices} is the dimension of the recession cone for
every face.

There are several Hasse diagrams in \polymake associated to a
polyhedral complex:
\begin{enumerate}
\item The \texttt{BOUNDED\_COMPLEX.HASSE\_DIAGRAM} collects only the bounded
faces.
\item The \texttt{HASSE\_DIAGRAM} is the full Hasse diagram, including far
faces.
\item The \texttt{COMPACTIFICATION} is the Hasse diagram of the tropical
compactification as described in this paper.
\end{enumerate}
On each of these Hasse diagrams one can consider cellular (co-)sheaves

\begin{example}
We consider the polyhedral complex consisting of the positive $x$-axis.
Its compactification will have one additional vertex at infinity.
\begin{lstlisting}
polytope > application "fan";

fan > $pc = new PolyhedralComplex(POINTS=>[[1,0],[0,1]], INPUT_POLYTOPES=>[[0,1]]);

fan > print $pc->COMPACTIFICATION->ADJACENCY;
{1 2}
{3}
{3}
{4}
{}

fan > print $pc->COMPACTIFICATION->DECORATION;
({} 0 {} {})
({0} 1 {0 1} {1})
({1} 1 {0} {})
({0 1} 2 {0 1} {})
({-1} 3 {-1} {})

\end{lstlisting}

In the \texttt{ADJACENCY}, the $i$-th row contains a list of the neighbors of
the $i$-th node. The \texttt{DECORATION} has four entries:
\begin{enumerate}
\item A set of integers $S$, the indices of the vertices forming the associated face of $\compact{\pc}$.
\item The rank of the face (to get the dimension subtract one).
\item The realisation $\realisation(S)$ as indices of vertices of $\pc$.
\item The sedentarity $\sedentarity(S)$ as indices of the rays of $\pc$.
\end{enumerate}
We see that node with decoration \texttt{(\{0\} 1 \{0 1\} \{1\})} is our new
vertex at infinity.
\end{example}

\begin{figure}
\includegraphics[scale=.3]{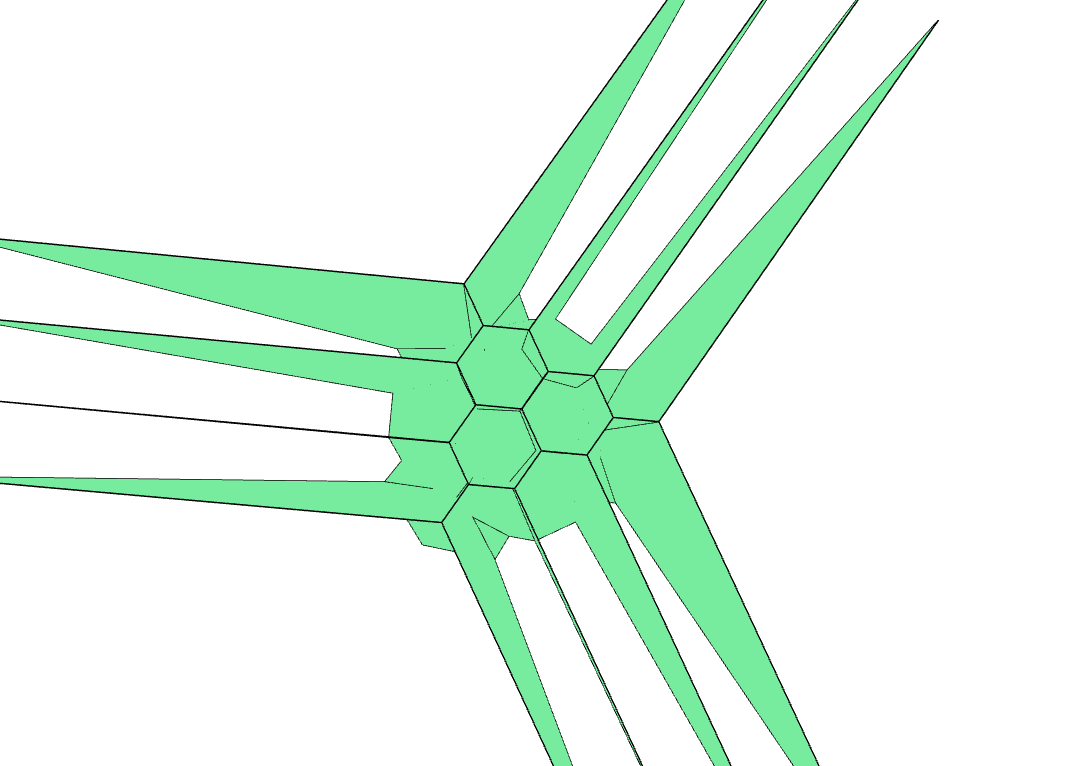}
\caption{Compactified tropical $K3$-surface}\label{fig:tropK3}
\end{figure}

Cellular (co-)sheaves in \polymake are realised as \texttt{EdgeMap}s on the
Hasse diagram. An \texttt{EdgeMap} is a map from the edges of a graph to some
category. In our case, edges are mapped to maps of vector spaces, represented
by matrices. 
 Our
extension for cellular sheaves can be found on github at
\url{https://github.com/lkastner/cellularSheaves}. 
In its \texttt{demo} folder
there are several \jupyter notebooks with commented examples on how to compute
cohomology of cellular sheaves. Since the code is too long to display here, we
will just briefly outline some examples.

\begin{example}

The compactification of  matroid fans explained in {\autoref{example:matroid} can be computed in \polymake}\label{example:matroid:code}.

We revisit Example~5 of \cite{ourTropHom} of the matroid of the so-called braid
arrangement of lines in $\CP^2$ , whose complement is the moduli space of
5-marked genus 0 curves $\mathcal{M}_{0,5}$, see \cite{ArdilaKlivans}. This matroid is also the graphical matroid of the complete graph $K_4$. 
Below is the polymake code which produces the matroid fan and its compactification. 

\begin{lstlisting}
application "fan";
$g = graph::complete(4);
$m = matroid::matroid_from_graph($g);
$t = tropical::matroid_fan<Max>($m);
$t->VERTICES;
$matFan = new PolyhedralComplex($t);
$matFanComp = $matFan->COMPACTIFICATION; 
print $matFanComp->nodes_of_rank(1)->size;

> 26
\end{lstlisting}
One can see that the compactification has $26 = 1+15+10$ vertices. 
The fan structure computed by $\polymake$ in this case is the coarsest structure 
of this fan, which corresponds to the minimal nested set compactification in the sense of 
\cite{FeichtnerSturmfels}. 
Note that vertices correspond to nodes of rank $1$ by \polymake's projective
viewpoint, i.e. the vertices of a polyhedral complex are the rays of the fan
one gets by embedding said complex at height one. 
The following code gives a full comparison of the number of faces of fixed dimension, in other words the F-vectors, of the
compactification and the original polyhedral complex:
\begin{lstlisting}
for(my $i=1; $i<$matFanComp->rank; $i++){
    print $matFanComp->nodes_of_rank($i)->size," ";
}
print "\n";
$matFanHasse = $matFan->HASSE_DIAGRAM;
$far = $matFan->FAR_VERTICES;
for(my $i=1; $i<$matFanHasse->rank; $i++){
    my @faces = @{$matFanHasse->nodes_of_rank($i)};
    @faces = map($matFanHasse->FACES->[$_], @faces);
    @faces = grep(($_*$far)->size < $_->size, @faces);
    print scalar @faces," ";
}

> 26 40 15 
> 1 10 15 
\end{lstlisting}
Note that due to \polymake considering a polyhedral complex as a fan, one gets
faces consisting only of far vertices. To get to the actual F-vector, these
have to be removed.

We can then use a loop to assemble the cosheaves for the tropical homology  on the compactification,
build their associated chain complexes and finally to compute their dimensions. 
\begin{lstlisting}
@rows = ();
for(my $i=0; $i<=$matFan->DIM; $i++){
    my $f = $matFan->compact_fcosheaf($i);
    my $d = build_full_chain($matFanComp, $matFanComp->ORIENTATIONS, $f->BLOCKS, false);
    push @rows, new Vector<Int>(topaz::betti_numbers($d));
}
print new Matrix(\@rows);

> 1 0 0
> 0 5 0
> 0 0 1
\end{lstlisting}
We see from the calculation that the tropical homology groups of the compactified matroid fan have the same Betti numbers as $\overline{\mathcal{M}_{0,5}}$, which is the blow up of $\CP^2$ in four points. Therefore, the compactification we have computed is the
 minimal nested set compactification in the sense of \cite{FeichtnerSturmfels}. The Chow group a of matroid with respect to a chosen nested set compactification is defined in   \cite{FeichtnerYuzvinsky}. Moreover, this can be generalised to any simplicial fan whose support is a matroid fan and the resulting Chow ring will satisfy Poincar\'e duality, Hard Lefschetz, and the Hodge Riemann bilinear relations \cite{amini_cubical}. 

The tropical homology of the fan prior to compactifying was computed in Example 5 of   \cite{ourTropHom}. This computation provides the duals of the  graded pieces of
the Orlik-Solomon algebra of this matroid. 

\end{example}

\begin{example}[Hodge numbers of a K3]\label{exam:k3}
\autoref{fig:tropK3} shows a compactified  K3 surface $\overline{X}$ in the
tropical toric variety $\mathbb{T} P^3$. The boundary of $\overline{X}$
consists of $4$ quartic tropical curves, one corresponding to each face of the
size $3$ standard simplex. The dimensions of the tropical homology
groups correspond to the Hodge numbers of a complex K3 surface.
Computing the dimensions of the tropical homology groups on the non-compact polyhedral complex
one arrives at
\begin{lstlisting}
0 0 34
0 31 3
1 0 1
\end{lstlisting}
or its transpose.
On the compactification, we arrive at the proper Hodge diamond:
\begin{lstlisting}
> print hodge_numbers($k3);
1 0 1
0 20 0
1 0 1
\end{lstlisting}
Of course this Hodge diamond has been known for some time, this example just
serves to give a glimpse at possible future computations.
\end{example}

\begin{remark}
In \autoref{exam:k3} the computation of a signed incidence relation was
already done in the background. In \polymake this is realised as an
\texttt{EdgeMap} on the Hasse diagram, labeling every edge with $\pm 1$. One
can access this property as \texttt{ORIENTATIONS} on both the
\texttt{HASSE\_DIAGRAM} and the \texttt{COMPACTIFICATION}. Due to the encoding
it is not trivial to make sense of the output. The nodes of the different Hasse
diagrams are numbered, the same is true for the edges, so to go backwards one
first needs to translate the index of an edge into its endpoints, and then
these endpoints back into faces.
\end{remark}

\begin{example}\label{example:betti}
In \cite{betti} the $\mathbb{Z}_2$-Betti numbers of the real part of a
hypersurface in a non-singular toric variety obtained  by a primitive
patchworking are equal to the Betti numbers of the sign cosheaf on the
associated tropical variety equipped with a real phase structure. The main
result of \cite{betti} is to bound the Betti numbers of the real part of the
hypersurface by sums of dimensions of the tropical homology groups.

These arguments apply to hypersurfaces in the torus  and partially compactified
(or compact) toric varieties. In the partially compactified (or compactified)
case one must work with the homology of the sign cosheaf on the closure of the
tropical variety.  The extension of the sign cosheaf to the compactification of
tropical hypersurfaces in toric varieties has also  been implemented in our
extension. The following is an example of a degree three curve in two-dimensional
tropical projective space, also using \polymake's patchworking framework
\cite{paul}.

\begin{lstlisting}
$g = toTropicalPolynomial("min(3*x_0,2*x_0+x_1,2*x_0+x_2,927+x_0+2*x_1,351+x_0+x_1+x_2,30+x_0+2*x_2,2856+3*x_1,1884+2*x_1+x_2,942+x_1+2*x_2,411+3*x_2)");
$trop = new Hypersurface<Min>(POLYNOMIAL=>$g);
# Print the monomials, as they might get reordered inside the hypersurface.

print $trop->COMPACTIFICATION->DECORATION;
> ({} 0 {} {})
> ({0} 1 {0 11} {0})
> ({1} 1 {0 8} {0})
> ({2} 1 {2 10} {2})
> ({3} 1 {0 9} {0})
> ({4} 1 {1 9} {1})
> ({5} 1 {2 4} {2})
> ({6} 1 {1 6} {1})
> ({7} 1 {2 3} {2})
> ({8} 1 {1 3} {1})
> ({9} 1 {11} {})
> ({10} 1 {8} {})
> ({11} 1 {10} {})
> ({12} 1 {9} {})
> ({13} 1 {7} {})
> ({14} 1 {5} {})
> ({15} 1 {6} {})
> ({16} 1 {4} {})
> ({17} 1 {3} {})
> ({0 9} 2 {0 11} {})
> ({1 10} 2 {0 8} {})
> ({2 11} 2 {2 10} {})
> ({3 12} 2 {0 9} {})
> ({4 12} 2 {1 9} {})
> ({5 16} 2 {2 4} {})
> ({6 15} 2 {1 6} {})
> ({7 17} 2 {2 3} {})
> ({8 17} 2 {1 3} {})
> ({9 11} 2 {10 11} {})
> ({9 13} 2 {7 11} {})
> ({10 12} 2 {8 9} {})
> ({10 13} 2 {7 8} {})
> ({11 14} 2 {5 10} {})
> ({13 15} 2 {6 7} {})
> ({14 15} 2 {5 6} {})
> ({14 16} 2 {4 5} {})
> ({16 17} 2 {3 4} {})
> ({-1} 3 {-1} {})

$pw = $trop->PATCHWORK(SIGNS=>[1,0,1,1,1,1,1,1,1,1]);
$cosheaf = $pw->sign_cosheaf();
$comp = $trop->COMPACTIFICATION;
$chain = fan::build_full_chain($comp, $comp->ORIENTATIONS, $cosheaf, true);
print topaz::betti_numbers<GF2>($chain);
> 2 2
\end{lstlisting}
\end{example}

\begin{example}[Smooth tropical cubic]
A smooth tropical cubic corresponds to a regular unimodular triangulation of
the dilated simplex $3\cdot\Delta_3$. Just as in \autoref{example:betti} one can
compute the Betti numbers of the real cubic via patchworkings and cellular
sheaves:
\begin{lstlisting}
> print topaz::betti_numbers<GF2>($chain);
1 7 1
\end{lstlisting}
\end{example}
\printbibliography

\end{document}